\documentclass[english,10pt]{article}
\usepackage{stmaryrd}
\usepackage{tipa}
\usepackage{tikz}
\usepackage{amsmath}
\usepackage{amssymb}
\usepackage{pst-node}
\usepackage{color}
\usepackage{extarrows}
\usepackage{setspace}
\usepackage{appendix}
\usepackage{enumitem}

\newtheorem{theorem}{Theorem}[section]
\newtheorem{corollary}[theorem]{Corollary}

\newtheorem{conjecture}[theorem]{Conjecture}
\newtheorem{observation}[theorem]{Observation}

\newtheorem{lemma}[theorem]{Lemma}

\newenvironment{proof}{\noindent {\bf Proof.}}{\rule{3mm}{3mm}\par\medskip}

\newcommand{\ZT}{\mbox{$\mathbb Z_3$}}

\newcommand{\ST}{\mbox{$\mathcal S_3$}}
\newcommand{\GS}{\mbox{$\mathcal{GS}$}}

\newcommand{\SST}{\mbox{$\langle \ST \rangle$}}

\begin{document}
\title{Realizing degree  sequences  with $\mathcal S_3$-connected graphs }

 \author{
Rui Guan$^1$, Chenglin Jiang$^1$, Hong-Jian Lai$^2$, Jiaao Li$^1$, Xinyuan Li$^1$\\
\small $^1$School of Mathematical Sciences and LPMC, Nankai University, Tianjin 300071, China \\
\small $^2$Department of Mathematics, West Virginia University, Morgantown, WV 26506, USA\\
\small Emails: gr@mail.nankai.edu.cn; jcl@mail.nankai.edu.cn; hjlai2015@hotmail.com;\\ \small lijiaao@nankai.edu.cn; xinyuanli@mail.nankai.edu.cn
 }

 \date{}
\maketitle

\begin{abstract}
A graph $G$ is \ST-connected if, for any mapping $\beta : V (G) \mapsto \ZT$ with $\sum_{v\in V(G)} \beta(v)\equiv 0\pmod3$, there exists a strongly connected  orientation $D$ satisfying $d^{+}_D(v)-d^{-}_D(v)\equiv \beta(v)\pmod{3}$  for any $v \in V(G)$. It is known that \ST-connected graphs are contractible configurations for the property of flow index strictly less than three.
In this paper, we provide a complete characterization of graphic sequences that have an $\mathcal{S}_{3}$-connected realization:  A graphic sequence $\pi=(d_1,\, \ldots,\, d_n )$ has an \ST-connected realization if and only if $\min \{d_1,\, \ldots,\, d_n\} \ge 4$ and $\sum^n_{i=1}d_i \ge 6n - 4$. Consequently, every graphic sequence $\pi=(d_1,\,  \ldots,\, d_n )$ with $\min \{d_1,\,  \ldots,\, d_n\} \ge 6$ has a realization $G$ with flow index strictly less than three. This supports a conjecture of Li, Thomassen, Wu and Zhang [European J. Combin., 70 (2018) 164-177] that every $6$-edge-connected graph has flow index strictly less than three. 
 \\[2mm]
\textbf{Keywords:} degree sequence; integer flow; flow index; \ST-connectivity

%\\[2mm] \textbf{AMS Subject Classification (2010):} 05C21,05C15, 05C20
\end{abstract}

\section{Introduction}
Graphs studied here are finite and may have multiple edges but no loops. We refer to a graph as simple if it contains neither multiple edges nor loops. Let $G=(V(G),E(G))$ be a graph with an orientation $D$. 
 For a vertex $v\in V(G)$, we use $E^{+}_D(v)$ (or $E^{-}_D(v)$, respectively) to denote the set of edges with tails (or heads, respectively) at $v$, and use $d^{+}_D(v)$ (or $d^{-}_D(v)$, respectively) to denote their sizes.  The subscript $D$ may be omitted when $D$ is understood from the context. Terms and notations not defined here are referred to \cite{BM1976}.

A function $\beta:V(G)\mapsto \ZT$ is called a $\ZT$ boundary function if $\sum_{v\in V(G)}\beta(v)\equiv 0\pmod{3}$. A graph $G$ is $\ZT$-connected if, for any $\ZT$ boundary function $\beta$ of $G$, there exists an orientation $D$ of $G$ such that $d^{+}_D(v)-d^{-}_D(v) \equiv \beta(v) \pmod{3}$ for any $v\in V(G)$. It is well-known that a graph admits a nowhere-zero $3$-flow if and only if it admits a modulo $3$-orientation  \cite{LLZ2015, Z1997}, i.e., an orientation such that the indegree is congruent to the outdegree modulo $3$ for each vertex.  
As mentioned in \cite{JLPT1992}, $\ZT$-connectivity is a generalization of the concept of nowhere-zero $3$-flows and serves as a potent tool in the study of nowhere-zero $3$-flows. 
It is proved by Lov\'asz, Thomassen, Wu and Zhang in \cite{LTWZ2013} that every 6-edge-connected graph is $\ZT$-connected and therefore admits a nowhere-zero $3$-flow.

In \cite{LTWZ2018}, the concept of  strongly connected modulo $3$-orientation is introduced to study the property of  flow index strictly less than three. 
An orientation is strongly connected if, for any two vertices $x$, $y \in V(G)$, there is a directed path from $x$ to $y$. It is shown in \cite{LTWZ2018} that a graph has flow index strictly less than three if and only if it has a strongly connected modulo $3$-orientation.
Similar to the definition of $\ZT$-connectivity, a graph $G$ is defined to be \ST-connected  if, for any $\ZT$ boundary function $\beta$ of $G$, there is a strongly connected orientation $D$ of $G$ such that $d^{+}_D(v)-d^{-}_D(v)\equiv\beta(v)\pmod 3$ for any $v\in V(G)$. Note that $\ST$-connected graphs are contractible configurations for the property of flow index strictly less than three (see \cite{HLLW2020, LLW2020}).
As an expected improvement of the result of \cite{LTWZ2013}, Li, Thomassen, Wu and Zhang \cite{LTWZ2018} proposed the following
conjecture.
\begin{conjecture}[\cite{LTWZ2018}]\label{CONJ:6flowstronglyconnected}
   Every 6-edge-connected graph has a strongly connected modulo $3$-orientation.
\end{conjecture}

Note that $K_6$ does not have a strongly connected modulo 3-orientation (see \cite{LTWZ2018}), and therefore the edge-connectivity condition cannot be relaxed in the conjecture. It is  proved in \cite{LTWZ2018} that Conjecture \ref{CONJ:6flowstronglyconnected} is true for 8-edge-connected graphs. 
Additionally, Conjecture \ref{CONJ:6flowstronglyconnected} was also verified for some families of graphs in \cite{HLLW2020}.

%%More introduction and results about $\ZT$-connected and $\mathcal{S}_{3}$-connected refer to \cite{LTWZ2013} and \cite{HLLW2020}\cite{LLW2020}.

An integer-valued sequence $\pi=(d_{1},\, \ldots,\, d_{n})$ is graphic if there is a  nontrivial simple graph $G$ with degree sequence $\pi$. In this case, the graph $G$ is referred to as a realization of $\pi$. 
Let $\GS$ be the set of all integer-valued non-increasing graphic sequences.
Let $\overline{\pi}=(n-1-d_{1},\,  \ldots,\, n-1-d_{n})$ be the complementary sequence of $\pi$. 
Obviously, $\overline{\pi}$ is graphic if and only if  $\pi$ is graphic. 
If an integer-valued non-increasing sequence $\pi$ has a realization $G$ that is simple and \ST-connected (or \ZT-connected, respectively), then we say that $G$ is an \ST-connected (or a \ZT-connected, respectively) realization of $\pi$, denoted by $\pi \in \GS(\ST)$ (or $\pi \in \GS(\ZT)$, respectively). 
Clearly, we can reorder each graphic sequence to obtain a non-increasing one without affecting any graph properties.
Therefore, for the remainder of this paper, we always assume that all graphic sequences are non-increasing unless otherwise specified.
For simplicity, we use exponents to denote degree multiplicities in a graphic sequence. For example, we write $(6^{3},\, 5^{4})$ to represent $(6,\, 6,\, 6,\, 5,\, 5,\, 5,\, 5)$.

The question of characterizing degree sequences with realizations
that are $\ZT$-connected or have nowhere-zero $3$-flows has been well-studied. 
Solving the open problem posed by 
Archdeacon \cite{A}, Luo, Xu, Zang and Zhang \cite{LXZZ2008} provided a complete characterization of graphic sequences with realizations having nowhere-zero $3$-flows.
\begin{theorem}[\cite{LXZZ2008}]
    Let $\pi=(d_{1},\,  \ldots,\, d_{n})$ be a graphic sequence with $d_1 \ge \cdots \ge d_n \ge 2$.
    Then the sequence $\pi$ has a realization that admits a nowhere-zero $3$-flow  
    if and only if $\pi \neq (3^4,\, 2)$, $(k,\, 3^k)$, $(k^2,\,  3^{k-1})$,
    where $k$ is an odd integer.
\end{theorem}

Furthermore, in \cite{LXZZ2008} and \cite{YLL2014}, the authors proposed the question of characterizing all graphic
sequences with $\ZT$-connected realizations, and this question was finally solved by Dai and Yin \cite{DY2016}. For $n\geq5$, let $S_{1}(n)=\{((n-1)^{2},\, 3^{n-k-2},\, 2^{k})~|~0\leq k \leq n-4,\,k\equiv n \pmod{2}\}$, and let $S_{2}(n)=\{(d_{1},\, d_{2},\, d_{3},\, d_{4},\, 2^{n-4})~|~n-1\geq d_{1}\geq d_{2}\geq d_{3}\geq d_{4}\geq 3,\,  d_{1}+d_{2}+d_{3}+d_{4}=2n + 4\}$. Denote 
\begin{align*}
 R(n)=
 \begin{cases}
     S_{1}(n)\cup S_{2}(n),\, &~\text{if $n$ is odd;}\\
     S_{1}(n)\cup S_{2}(n)\cup \{(n-1,\, 3^{n-1})\},\, 
  &~\text{if $n$ is even.}
 \end{cases}
\end{align*}
\begin{theorem}
[\cite{DY2016}]\label{THM: Z3connectedsequence}
    Let $n\geq5$, and let $\pi=(d_{1},\,  \ldots,\, d_{n})$ be a graphic sequence with $d_1   \ge \cdots \ge d_n \ge 2$. Then the sequence $\pi$ has a realization that is $\ZT$-connected if and only if  $\sum_{i=1}^n d_i\geq 4n-4$ and $\pi\not\in R(n)$, where $R(n)$ is the well-characterized set defined above. 
\end{theorem}

Similar to the works mentioned above, it is natural to study the degree sequence realization problem concerning the property of flow index strictly less than three or $\ST$-connectivity.
Motivated by Conjecture \ref{CONJ:6flowstronglyconnected}  and the aforementioned studies, this paper presents a complete characterization of graphic sequences with $\mathcal{S}_{3}$-connected realizations. 

\begin{theorem}\label{THM: S3connectedsequence}
Let $\pi=(d_{1},\,  \ldots,\, d_{n})$ be a graphic sequence with $d_1  \ge \cdots \ge d_n > 0$.
Then the sequence $\pi$ has a realization that is \ST-connected if and only if $\sum_{i=1}^n d_i \geq 6n-4$ and $d_n \ge 4$.
\end{theorem}

Consequently, Theorem \ref{THM: S3connectedsequence} implies the following corollary, which provides some supports for Conjecture \ref{CONJ:6flowstronglyconnected}.
\begin{corollary}
    Every graphic sequence $\pi=(d_1,\,   \ldots,\,  d_n )$ with $\min \{d_1,\,   \ldots,\,  d_n\} \ge 6$ has a realization with a strongly connected modulo $3$-orientation.
\end{corollary}

The remainder of this paper is organized as follows. 
In Section 2, we state some results that are used in the following proofs. 
In Section \ref{SEC: 2large}, we handle some special cases and show that $\pi \in \GS(\ST)$ when there are
at least two high degrees.
In Section \ref{SEC: boundary},  we discuss the situation in which graphic sequences cannot be reduced in order by applying laying and lifting operations.
In Section \ref{SEC: main theorem}, we prove the main theorem, which fully characterizes all
graphic sequences that belong to $\GS(\mathcal{S}_3)$. Furthermore, the main results of this paper will be served as a tool for characterizing all graphic sequences that have strongly connected modulo $3$-orientation realizations in future studies.

\section{Preliminaries}
In this section, we present some foundational lemmas and special graphs required to prove the main theorem.

\subsection{Some \ST-connected graphs and the contraction methods}

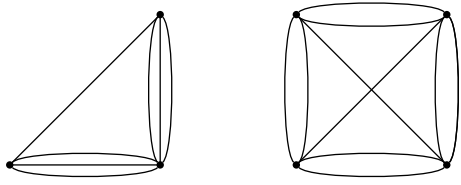
\begin{figure}[ht]
    \centering
\begin{tikzpicture}[scale =0.2]
 \tikzstyle{transition}=[circle,draw=black,fill=black,thick,
	inner sep=0pt, minimum  size=0.7mm]
  \tikzstyle{firstedge}=[line width = 0.5pt,color=black]
		\node [transition] (v1) at (5, 5) {};
	\node [transition] (v3) at (-5, -5){};
    \node [transition] (v4) at (5,-5){};
    \draw [firstedge] plot[smooth, tension=.7] coordinates {(v1) (v3)};
	\draw [firstedge] plot[smooth, tension=.7] coordinates {(v1) (v4)};
	\draw [firstedge] plot[smooth, tension=.7] coordinates {(v3) (v4)};
      \draw[firstedge][out=315,in=45,looseness = 0.3] (v1) to(v4);
      \draw[firstedge][out=225,in=135,looseness = 0.3] (v1) to(v4);
      \draw[firstedge][out=-45,in=225,looseness = 0.3] (v3) to(v4);
      \draw[firstedge][out=45,in=135,looseness = 0.3] (v3) to(v4);
      \put(-0.5,-2.9){ $K_{(1,3,3)}$ }
	\end{tikzpicture}\quad \quad \quad \quad 
     \begin{tikzpicture}[scale=0.2]
 \tikzstyle{transition}=[circle,draw=black,fill=black,thick,
	inner sep=0pt, minimum  size=0.7mm]
  \tikzstyle{firstedge}=[line width = 0.5pt,color=black]
	\node [transition] (v1) at (5,5) {};
	\node [transition] (v2) at (-5,5) {};
	\node [transition] (v3) at (-5,-5){};
    \node [transition] (v4) at (5,-5){};
    \draw [firstedge] plot[smooth, tension=.7] coordinates {(v1) (v3)};
    \draw [firstedge] plot[smooth, tension=.7] coordinates {(v2) (v4)};
     \draw[firstedge][out=135,in=45,looseness = 0.3] (v1) to(v2);
     \draw[firstedge][out=-135,in=-45,looseness = 0.3] (v1) to(v2);
      \draw[firstedge][out=315,in=45,looseness = 0.3] (v1) to(v4);
      \draw[firstedge][out=225,in=135,looseness = 0.3] (v2) to(v3);
      \draw[firstedge][out=-45,in=45,looseness = 0.3] (v2) to(v3);
      \draw[firstedge][out=315,in=45,looseness = 0.3] (v1) to(v4);
      \draw[firstedge][out=225,in=135,looseness = 0.3] (v1) to(v4);
      \draw[firstedge][out=-45,in=-135,looseness = 0.3] (v3) to(v4);
      \draw[firstedge][out=45,in=135,looseness = 0.3] (v3) to(v4);
      \put(-0.5,-2.9){ $K_4^*$ }
	\end{tikzpicture}
    \caption{The graphs $K_{(1,3,3)}$ and $K_4^*$.}
    \label{FIG: T133_K4*}
\end{figure}

Let $\SST$ (or $\langle \ZT \rangle$, respectively) be the family of graphs that are $\ST$-connected (or $\ZT$-connected, respectively). At first, we list some special graphs that belong to $\SST$, as they are crucial to our proof.
Let $K_n$ represent the complete graph on $n$ vertices. 
Let $mK_2$ denote the graph with two vertices and $m$ parallel edges. 
The graphs $K_{(1, 3, 3)}$ and $K_4^*$ are defined to be the graphs respectively depicted in Figure \ref{FIG: T133_K4*}.

\begin{lemma}[\cite{LLW2020}]\label{LEM: S3graph}
     Each of the following holds:
     \begin{enumerate}[label=(\roman*)]
         \item $K_{n}\in \SST$ if and only if $n\geq7$.
         \item $mK_{2}\in \SST$ if and only if $m\geq4$.
         \item $K_{(1,3,3)}$, $K_4^*\in \SST$.
     \end{enumerate}
    \end{lemma}

As observed in \cite{HLLW2020}, $\ST$-connected graphs and $\ZT$-connected graphs are closely related. Specifically, by adding an arbitrary Hamiltonian cycle to any $\ZT$-connected graph, we can construct an $\ST$-connected graph. The presence of the Hamiltonian cycle ensures that the resulting graph is strongly connected while preserving the boundary of each vertex.

\begin{lemma}[\cite{HLLW2020}]\label{LEM: Z_3+H=S3}
    Let $G$ be a graph with a Hamiltonian cycle $C$.
    If $G-E(C)$ is $\ZT$-connected, then $G\in \SST$.
\end{lemma}

For a graph $G$ with a vertex $u$ of degree 4 or higher, if there are two distinct vertices $v$ and $w$ adjacent to $u$, we define $G_{[u,\, vw]}=G-u+vw$ as the graph obtained from $G$ by removing the vertex $u$ along with all its incident edges, and adding a new edge $vw$. This operation is called \textit{lifting}. For convenience, the lifting process is denoted as $G \rightarrow G_{[u,vw]}$.

\begin{lemma}[\cite{HLLW2020}]\label{LEM: liftinggraph}
Let $G$ be a  graph with a vertex $u$,  and let $uv$ and $uw$ be two edges in $E(G)$.
If $d_G(u) \ge 4$ and $G_{[u,vw]}\in \SST$, then $G\in  \SST $.
\end{lemma}

For an edge set $E' \subseteq E(G)$, the contraction $G/E'$ is the graph obtained from $G$ by identifying the two ends of each edge in $E'$, and then removing the resulting loops. 
For convenience, $G/E(G')$ is denoted by $G/G'$ if $G'$ is a subgraph of $G$. 
The following lemma is very powerful for determining whether a graph belongs to $\SST$.

\begin{lemma}[\cite{LLW2020}]\label{LEM: contract graph}
    %Each of the following holds:\\
    %(i) If $G\in \langle \ST \rangle$ and $e\in E(G)$,\,  then $G/e\in \SST$.\\
   Let $G$ be a graph with a subgraph $G'$. If $G'\in \SST$ and $G/G'\in \SST$, then $G\in \SST$. \end{lemma}

If $G'$ is a subgraph of $G$ and there exist two distinct vertices $v$ and $w$ in $V(G')$ such that there is a $(v,\, w)$-path in $G-E(G')$, then we say that $G'$ is a proper subgraph of $G$. 
The following lemma shows
that $\ST$-connectivity can be inherited by contracting a proper subgraph $K_6$.

\begin{lemma}[\cite{HLLW2020}]\label{LEM: contractK6 }
    Let $G$ be a graph. If $K_6$ is a proper subgraph of $G$ and $G/K_6 \in \SST$, then  $G \in \SST$.
\end{lemma}

\subsection{Graphic sequences and \ST-preserving operations}
In the inductive arguments of our proofs, a critical part is determining whether the resulting degree sequences still satisfy the conditions of the induction hypothesis after certain operations. The following useful theorems due to Erd\H{o}s and Gallai \cite{EG1960} and Hakimi \cite{H1962, H1963} can  be applied to verify whether a degree sequence is graphic.

\begin{theorem}[Erd\H{o}s-Gallai Theorem \cite{EG1960}]\label{THM: verify GS-iff condition}
Let $\pi=(d_1,\,  \ldots,\, d_n)$ be an integer-valued sequence, where $n-1 \ge d_1 \geq   \cdots \ge d_n \ge 0$ and $\sum_{i =1}^n d_i$ is even,
and let $f(\pi)=\max \{i~|~d_{i}\geq i,\,  1\le i\le n\}$. 
Then the degree sequence $\pi$ is graphic if and only if $$\sum_{i=1}^{k}d_{i}\leq k(k-1)+\sum_{i=k+1}^{n}\min \{k,\, d_{i}\}$$ for each integer $k$ with $1\leq k \leq f(\pi)$. 
\end{theorem}
Let $\pi=(d_1,\,  \ldots,\, d_n)$ be an integer-valued sequence with $n-1 \ge d_1  \ge \cdots \ge d_n \ge 0$. 
We define that the sequence $(d_1-1,\,  \ldots,\,  d_{d_n}-1,\,  d_{d_n + 1},\,  \ldots,\,  d_{n-1})$ as the resulting sequence obtained from $\pi$ by laying off $d_n$. Additionally, we introduce a \textit{laying sequence}
$\pi'= (d'_1,\,  \ldots,\, d'_{n-1})$ defined as the non-increasing reordered version of $(d_1-1,\,  \ldots,\,  d_{d_n}-1,\,  d_{d_n + 1},\,  \ldots,\,  d_{n-1})$.

%This sequence is obtained by deleting  $d_n$ and reducing the first $d_n$ degrees by $1$ in the original sequence $\pi$. 
%For convenience, we refer to this operation as \textit{laying} or laying off a vertex of degree $d_n$.

\begin{theorem}[Hakimi \cite{H1962, H1963}]\label{THM: vetifyGSdelete}
Let $\pi=(d_1,\,  \ldots,\, d_n)$ be an integer-valued sequence with $n-1 \ge d_1  \ge \cdots \ge d_n \ge 0$.
Then the degree sequence $\pi$ is graphic if and only if  the laying sequence $\pi'$ is graphic.
\end{theorem}

When the difference between the maximum and minimum degrees is small, the following lemma can be applied to determine whether a degree sequence is graphic.

\begin{lemma}[\cite{YL2003}]\label{LEM: verifyGSinequality}
    Let $\pi=(d_1,\,  \ldots,\, d_n)$ be an integer-valued sequence, where $n-1 \ge d_1  \ge \cdots \ge d_n > 0$ and $\sum_{i =1}^n d_i$ is even. If $n\ge \frac{1}{d_n}\lfloor \frac{(d_1+d_n+1)^2}{4}\rfloor$, then $\pi$ is graphic.
\end{lemma}

Similar to laying sequence, we define \textit{lifting sequence} $ \pi_{l} = (d^l_1,\, \ldots,\, d^l_{n-1})$ as the non-increasing reordered sequence of 
     $  (d_1-1,\, \ldots,\, d_{d_n-2}-1,\, d_{d_n-1},\, \ldots,\, d_{n-1})$. 
     %The lifting sequence $\pi_l$ is obtained from $\pi$ by lifting $d_n$.
%\subsection{\ST-Preserving Operations }
Now we show that the lifting operation can be used to obtain degree sequences that have $\ST$-connected realizations.
\begin{lemma}\label{LEM: sequenceinverselift}
    Let $\pi=(d_{1},\,   \ldots,\, d_{n})\in \GS$ with $ d_{n} \ge 4$. 
    If the lifting sequence $\pi_{l}  \in  \GS(\ST)$ and $$d_{1}+\cdots+d_{d_{n}-2}-(d_{n}-2)<d_{d_{n}-1}+\cdots+d_{n-1},$$ then $\pi \in \GS(\mathcal{S}_{3})$.
\end{lemma}
\begin{proof}
    Let $G_l$ be an $\mathcal{S}_{3}$-connected realization of the lifting sequence $\pi_l$. 
    For $1\le i \le d_n -2$, we denote $v_i$ as the vertex of $G_{l}$ with $d_{G_l}(v_i) = d_{i}-1$.
    Now we have $d_{1}+\cdots+d_{d_{n}-2}-(d_{n}-2)<d_{d_{n}-1}+\cdots+d_{n-1}$, and this inequality implies that there exists an edge $uw  \in E(G_l)$ such that $u,\, w \notin \{v_{1},\, \ldots,\, v_{d_{n}-2}\}$. 
    Then we introduce a new vertex $v_n$ and construct a new graph $G$ from $G_l$ by setting 
 $V(G)= V(G_l) \cup \{v_n\}$ and
     $E(G)= E(G_l) \cup \{v_{1}v_{n},\, \ldots,\, v_{d_{n}-2}v_{n},\, uv_{n},\, wv_{n}\} - \{uw\}$. 
     Clearly, $G$ is simple and its degree sequence is $\pi$. Since $G_{[v_{n},\, uw]}=G_l\in \SST$, by Lemma \ref{LEM: liftinggraph} we know that $G\in \SST$, i.e., $\pi \in \GS(\ST)$.
\end{proof}

Now we show that if the laying sequence $\pi'$ has an $\ST$-connected realization, then the original degree sequence $\pi$ also has one.
\begin{lemma}\label{LEM: sequence_addvertex}
 Let $\pi=(d_{1},\,   \ldots,\, d_{n})\in \GS$ with $ d_{n} \ge 4$. 
    If the laying sequence $\pi' \in \GS(\ST)$,
     then $\pi \in \GS(\ST)$.
\end{lemma}
\begin{proof}
    Let $G'$ be an $\ST$-connected realization of $\pi'$. For $1 \le i \le d_{n}$, we denote $v_i$ to be the vertex of $G'$  with degree $d_{i}-1$.
    Let $u$ be a new vertex.
    Now we construct a new graph $G$  such that $V(G) = V(G') \cup \{ u \}$ and $E(G) = E(G') \cup \{v_1u,\,  v_2u,\, \ldots,\, v_{d_n}u\}$.
    Obviously, $G$ is a realization of $\pi$.
    Since $G/G' = m K_2$ with $m \ge 4$, we have $G/G' \in \SST$ by Lemma \ref{LEM: S3graph} (ii). 
    Therefore, by Lemma \ref{LEM: contract graph}, $G \in \SST$, i.e., $\pi \in \GS(\ST)$.
\end{proof}

\subsection{Some special sequences with $\ST$-connected realizations}

In this subsection, we characterize some special degree sequences that have $\ST$-connected realizations.
In \cite{L2000,YLL2014}, the authors provided several $\ZT$-connected graphs  that can be used to construct  $\ST$-connected graphs by Lemma \ref{LEM: Z_3+H=S3}.
The graph $W_4$ is constructed by adding a new vertex to a $4$-cycle and connecting it to each vertex of the cycle (see Figure \ref{FIG: W4_4^5,3^4}).

\begin{figure}[ht]
    \centering
     \begin{tikzpicture}[scale =0.4]
	\tikzstyle{transition}=[circle,draw=black,fill=black,thick,
	inner sep=0pt, minimum  size=0.7mm]
	\tikzstyle{firstedge}=[line width = 0.5pt,color=black]
	\tikzstyle{secondedge}= [densely dashed]

	\node [transition] (v5) at (2,-2) {};
	\node [transition] (v6) at (-2,-2) {};
	\node [transition] (v7) at (-2,2) {};
    \node [transition] (v8) at (2,2) {};
    \node [transition] (v9) at (0,0) {};

    \draw [firstedge] plot[smooth, tension=.7] coordinates {(v5) (v6)};
    \draw [firstedge] plot[smooth, tension=.7] coordinates {(v6) (v7)};
    \draw [firstedge] plot[smooth, tension=.7] coordinates {(v7) (v8)};
    \draw [firstedge] plot[smooth, tension=.7] coordinates {(v8) (v5)};
    \draw [firstedge] plot[smooth, tension=.7] coordinates {(v5) (v9)};
    \draw [firstedge] plot[smooth, tension=.7] coordinates {(v6) (v9)};
    \draw [firstedge] plot[smooth, tension=.7] coordinates {(v7) (v9)};
    \draw [firstedge] plot[smooth, tension=.7] coordinates {(v8) (v9)};3
    \put(-0.4,-2.5){ $W_4$,}
\end{tikzpicture}
\quad \quad
    \begin{tikzpicture}[scale =0.4]
	\tikzstyle{transition}=[circle,draw=black,fill=black,thick,
	inner sep=0pt, minimum  size=0.7mm]
	\tikzstyle{firstedge}=[line width = 0.5pt,color=black]
	\tikzstyle{secondedge}= [densely dashed]

	\node [transition] (v1) at (4,4) {};
	\node [transition] (v2) at (-4,4) {};
	\node [transition] (v3) at (-4,-4) {};
	\node [transition] (v4) at (4,-4) {};
	\node [transition] (v5) at (2,-2) {};
	\node [transition] (v6) at (-2,-2) {};
	\node [transition] (v7) at (-2,2) {};
    \node [transition] (v8) at (2,2) {};
    \node [transition] (v9) at (0,0) {};

	\draw [firstedge] plot[smooth, tension=.7] coordinates {(v1) (v2)};
	\draw [firstedge] plot[smooth, tension=.7] coordinates {(v2) (v3)};
    \draw [firstedge] plot[smooth, tension=.7] coordinates {(v3) (v4)};
    \draw [firstedge] plot[smooth, tension=.7] coordinates {(v4) (v1)};
    \draw [firstedge] plot[smooth, tension=.7] coordinates {(v1) (v8)};
    \draw [firstedge] plot[smooth, tension=.7] coordinates {(v2) (v7)};
    \draw [firstedge] plot[smooth, tension=.7] coordinates {(v3) (v6)};
    \draw [firstedge] plot[smooth, tension=.7] coordinates {(v4) (v5)};
    \draw [firstedge] plot[smooth, tension=.7] coordinates {(v5) (v6)};
    \draw [firstedge] plot[smooth, tension=.7] coordinates {(v6) (v7)};
    \draw [firstedge] plot[smooth, tension=.7] coordinates {(v7) (v8)};
    \draw [firstedge] plot[smooth, tension=.7] coordinates {(v8) (v5)};
    \draw [firstedge] plot[smooth, tension=.7] coordinates {(v5) (v9)};
    \draw [firstedge] plot[smooth, tension=.7] coordinates {(v6) (v9)};
    \draw [firstedge] plot[smooth, tension=.7] coordinates {(v7) (v9)};
    \draw [firstedge] plot[smooth, tension=.7] coordinates {(v8) (v9)};
  \put(2.7,3.5){ $v_1$ }
    \put(-3.3,3.5){ $v_2$ }
    \put(-3.3,-3.7){ $v_3$ }
    \put(2.7,-3.7){ $v_4$ }
    \put(1,-2){ $v_5$ }
    \put(-1.6,-2){ $v_6$ }
    \put(-1.6,1.8){ $v_7$ }
    \put(1,1.8){ $v_8$ }
     \put(-0.3,0.5){ $v_9$ }
    \put(-1,-4){ $(4^5,\, 3^4)$,}
\end{tikzpicture}
\quad \quad
   \begin{tikzpicture}[scale =0.4]
	\tikzstyle{transition}=[circle,draw=black,fill=black,thick,
	inner sep=0pt, minimum  size=0.7mm]
	\tikzstyle{firstedge}=[line width = 0.5pt,color=black]
	\tikzstyle{secondedge}= [densely dashed]

	\node [transition] (v1) at (4,4) {};
	\node [transition] (v2) at (-4,4) {};
	\node [transition] (v3) at (-4,-4) {};
	\node [transition] (v4) at (4,-4) {};
	\node [transition] (v5) at (2,-2) {};
	\node [transition] (v6) at (-2,-2) {};
	\node [transition] (v7) at (-2,2) {};
    \node [transition] (v8) at (2,2) {};
    \node [transition] (v9) at (0,0) {};

	\draw [firstedge] plot[smooth, tension=.7] coordinates {(v1) (v2)};
	\draw [firstedge] plot[smooth, tension=.7] coordinates {(v2) (v3)};
    \draw [firstedge] plot[smooth, tension=.7] coordinates {(v3) (v4)};
    \draw [firstedge] plot[smooth, tension=.7] coordinates {(v4) (v1)};
    \draw [firstedge] plot[smooth, tension=.7] coordinates {(v1) (v8)};
    \draw [firstedge] plot[smooth, tension=.7] coordinates {(v2) (v7)};
    \draw [firstedge] plot[smooth, tension=.7] coordinates {(v3) (v6)};
    \draw [firstedge] plot[smooth, tension=.7] coordinates {(v4) (v5)};
    \draw [firstedge] plot[smooth, tension=.7] coordinates {(v5) (v6)};
    \draw [firstedge] plot[smooth, tension=.7] coordinates {(v6) (v7)};
    \draw [firstedge] plot[smooth, tension=.7] coordinates {(v7) (v8)};
    \draw [firstedge] plot[smooth, tension=.7] coordinates {(v8) (v5)};
    \draw [firstedge] plot[smooth, tension=.7] coordinates {(v5) (v9)};
    \draw [firstedge] plot[smooth, tension=.7] coordinates {(v6) (v9)};
    \draw [firstedge] plot[smooth, tension=.7] coordinates {(v7) (v9)};
    \draw [firstedge] plot[smooth, tension=.7] coordinates {(v8) (v9)};
    \draw [secondedge][out=-110,in=10,looseness = 1]  (v1)to (v3);
    \draw [secondedge][out=80,in=-160,looseness = 1]  (v3)to (v8);
     \draw [secondedge][out=-100,in=10,looseness = 1]  (v8)to (v6);
    \draw [secondedge][out=130,in=-80,looseness = 1]  (v6)to (v2);
    \draw [secondedge][out=-10,in=90,looseness = 1]  (v2)to (v9);
    \draw [secondedge][out=-80,in=160,looseness = 1.1]  (v9)to (v4);
    \draw [secondedge][out=100,in=-20,looseness = 1]  (v4)to (v7);
    \draw [secondedge][out=-80,in=170,looseness =1]  (v7)to (v5);
    \draw [secondedge][out=60,in=-95,looseness =1]  (v5)to (v1);
    
    \put(-1,-4){ $(6^5,\, 5^4)$,}
\end{tikzpicture}
\\ \quad

    \caption{The graphs $W_4$,  $(4^5,\, 3^4)$-realization, and $(6^5,\, 5^4)$-realization.}
    \label{FIG: W4_4^5,3^4}
\end{figure}
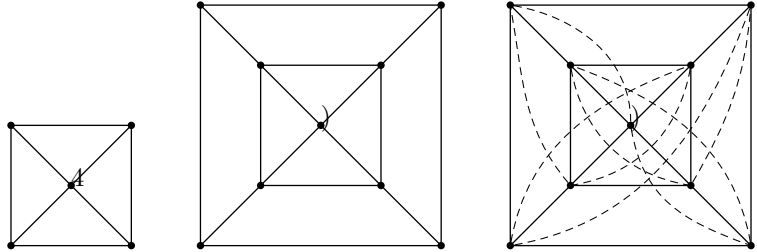

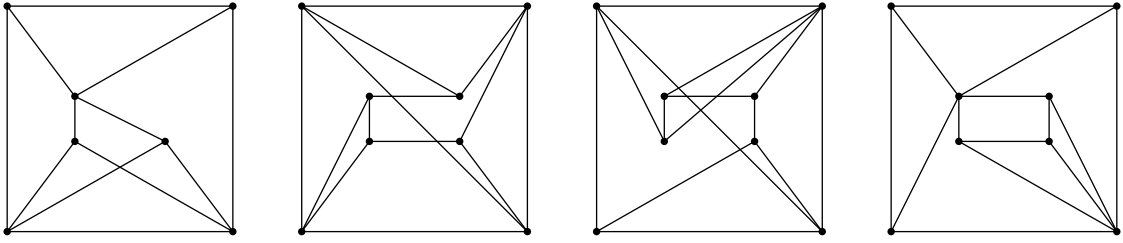
\begin{figure}[ht]
    \centering
    \begin{tikzpicture}[scale =0.3]
	\tikzstyle{transition}=[circle,draw=black,fill=black,thick,
	inner sep=0pt, minimum  size=0.7mm]
	\tikzstyle{firstedge}=[line width = 0.5pt,color=black]
	\tikzstyle{ AAge}= [densely dashed]

	\node [transition] (v1) at (5,5) {};
	\node [transition] (v2) at (-5,5) {};
	\node [transition] (v3) at (-5,-5) {};
	\node [transition] (v4) at (5,-5) {};
	\node [transition] (v5) at (2,-1) {};
	\node [transition] (v6) at (-2,-1) {};
	\node [transition] (v7) at (-2,1.) {};

	\draw [firstedge] plot[smooth, tension=.7] coordinates {(v1) (v2)};
	\draw [firstedge] plot[smooth, tension=.7] coordinates {(v1) (v4)};
	\draw [firstedge] plot[smooth, tension=.7] coordinates {(v1) (v7)};
	\draw [firstedge] plot[smooth, tension=.7] coordinates {(v2) (v3)};
	\draw [firstedge] plot[smooth, tension=.7] coordinates {(v2) (v7)};
	\draw [firstedge] plot[smooth, tension=.7] coordinates {(v3) (v4)};
	\draw [firstedge] plot[smooth, tension=.7] coordinates {(v3) (v5)};
	\draw [firstedge] plot[smooth, tension=.7] coordinates {(v3) (v6)};
	\draw [firstedge] plot[smooth, tension=.7] coordinates {(v4) (v5)};
	\draw [firstedge] plot[smooth, tension=.7] coordinates {(v4) (v6)};
	\draw [firstedge] plot[smooth, tension=.7] coordinates {(v5) (v7)};
	\draw [firstedge] plot[smooth, tension=.7] coordinates {(v6) (v7)};

    \put(-1,-3.9){ $(4^3,\, 3^4)$,}
\end{tikzpicture}
\quad \quad
    \begin{tikzpicture}[scale =0.3]
		\tikzstyle{transition}=[circle,draw=black,fill=black,thick,
		inner sep=0pt, minimum  size=0.7mm]
		\tikzstyle{firstedge}=[line width = 0.5pt,color=black]
		\tikzstyle{ AAge}= [densely dashed]

		\node [transition] (v1) at (5,5) {};
		\node [transition] (v2) at (-5,5) {};
		\node [transition] (v3) at (-5,-5) {};
		\node [transition] (v4) at (5,-5) {};
		\node [transition] (v5) at (2,-1) {};
		\node [transition] (v6) at (-2,-1) {};
		\node [transition] (v7) at (-2,1.) {};
		\node [transition] (v8) at (2,1) {};

		\draw [firstedge] plot[smooth, tension=.7] coordinates {(v1) (v2)};
		\draw [firstedge] plot[smooth, tension=.7] coordinates {(v1) (v4)};
		\draw [firstedge] plot[smooth, tension=.7] coordinates {(v1) (v5)};
		\draw [firstedge] plot[smooth, tension=.7] coordinates {(v1) (v8)};
		\draw [firstedge] plot[smooth, tension=.7] coordinates {(v2) (v3)};
		\draw [firstedge] plot[smooth, tension=.7] coordinates {(v2) (v4)};
		\draw [firstedge] plot[smooth, tension=.7] coordinates {(v2) (v8)};
		\draw [firstedge] plot[smooth, tension=.7] coordinates {(v3) (v4)};
		\draw [firstedge] plot[smooth, tension=.7] coordinates {(v3) (v6)};
		\draw [firstedge] plot[smooth, tension=.7] coordinates {(v3) (v7)};
		\draw [firstedge] plot[smooth, tension=.7] coordinates {(v4) (v5)};
		\draw [firstedge] plot[smooth, tension=.7] coordinates {(v5) (v6)};
		\draw [firstedge] plot[smooth, tension=.7] coordinates {(v6) (v7)};
		\draw [firstedge] plot[smooth, tension=.7] coordinates {(v7) (v8)};

\put(-1,-3.9){  $(4^4,\,  3^4)$}
	\end{tikzpicture}
 \quad \quad
\begin{tikzpicture}[scale =0.3]
	\tikzstyle{transition}=[circle,draw=black,fill=black,thick,
	inner sep=0pt, minimum  size=0.7mm]
	\tikzstyle{firstedge}=[line width = 0.5pt,color=black]
	\tikzstyle{ AAge}= [densely dashed]

	\node [transition] (v1) at (5,5) {};
	\node [transition] (v2) at (-5,5) {};
	\node [transition] (v3) at (-5,-5) {};
	\node [transition] (v4) at (5,-5) {};
	\node [transition] (v5) at (2,-1) {};
	\node [transition] (v6) at (-2,-1) {};
	\node [transition] (v7) at (-2,1.) {};
	\node [transition] (v8) at (2,1) {};

	\draw [firstedge] plot[smooth, tension=.7] coordinates {(v1) (v2)};
	\draw [firstedge] plot[smooth, tension=.7] coordinates {(v1) (v4)};
	\draw [firstedge] plot[smooth, tension=.7] coordinates {(v1) (v6)};
	\draw [firstedge] plot[smooth, tension=.7] coordinates {(v1) (v7)};
	\draw [firstedge] plot[smooth, tension=.7] coordinates {(v1) (v8)};
	\draw [firstedge] plot[smooth, tension=.7] coordinates {(v2) (v3)};
	\draw [firstedge] plot[smooth, tension=.7] coordinates {(v2) (v4)};
	\draw [firstedge] plot[smooth, tension=.7] coordinates {(v2) (v6)};
	\draw [firstedge] plot[smooth, tension=.7] coordinates {(v3) (v4)};
	\draw [firstedge] plot[smooth, tension=.7] coordinates {(v3) (v5)};
	\draw [firstedge] plot[smooth, tension=.7] coordinates {(v4) (v5)};
	\draw [firstedge] plot[smooth, tension=.7] coordinates {(v5) (v8)};
	\draw [firstedge] plot[smooth, tension=.7] coordinates {(v6) (v7)};
	\draw [firstedge] plot[smooth, tension=.7] coordinates {(v7) (v8)};

 \put(-1.5,-3.9){ $(5,\,  4^2,\,  3^5)$}
\end{tikzpicture}
\quad \quad
	\begin{tikzpicture}[scale = 0.3]
 \tikzstyle{transition}=[circle,draw=black,fill=black,thick,
	inner sep=0pt, minimum  size=0.7mm]
	\tikzstyle{firstedge}=[line width = 0.5pt,color=black]
		\node [transition] (v1) at (5,5) {};
		\node [transition] (v2) at (-5,5) {};
		\node [transition] (v3) at (-5,-5) {};
		\node [transition] (v4) at (5,-5) {};
		\node [transition] (v5) at (2,-1) {};
		\node [transition] (v6) at (-2,-1) {};
		\node [transition] (v7) at (-2,1.) {};
		\node [transition] (v8) at (2,1) {};
		\draw [firstedge] plot[smooth, tension=.7] coordinates {(v1) (v2)};
		\draw [firstedge] plot[smooth, tension=.7] coordinates {(v1) (v4)};
		\draw [firstedge] plot[smooth, tension=.7] coordinates {(v1) (v7)};
		\draw [firstedge] plot[smooth, tension=.7] coordinates {(v2) (v3)};
		\draw [firstedge] plot[smooth, tension=.7] coordinates {(v2) (v7)};
		\draw [firstedge] plot[smooth, tension=.7] coordinates {(v3) (v4)};
		\draw [firstedge] plot[smooth, tension=.7] coordinates {(v3) (v7)};
		\draw [firstedge] plot[smooth, tension=.7] coordinates {(v4) (v5)};
		\draw [firstedge] plot[smooth, tension=.7] coordinates {(v4) (v6)};
		\draw [firstedge] plot[smooth, tension=.7] coordinates {(v4) (v8)};
		\draw [firstedge] plot[smooth, tension=.7] coordinates {(v5) (v6)};
		\draw [firstedge] plot[smooth, tension=.7] coordinates {(v5) (v8)};
		\draw [firstedge] plot[smooth, tension=.7] coordinates {(v6) (v7)};
		\draw [firstedge] plot[smooth, tension=.7] coordinates {(v7) (v8)};
  \put(-1,-3.9){ $(5^2,\, 3^6)$ }
	\end{tikzpicture}
 \\ \quad
    \caption{The \ZT-connected realizations of $(4^3,\, 3^4)$, $(4^4,\,  3^4)$, $(5,\,  4^2,\,  3^5)$, and $(5^2,\, 3^6)$.}
    \label{FIG: Z3graph}
\end{figure}
\begin{lemma}\label{LEM: Z3realization} Each of the following holds:
\begin{enumerate}[label=(\roman*)]
    \item{\rm(Proposition 3.6 of \cite{L2000})} The graph $W_4$ is  $\ZT$-connected.
    \item {\rm(Lemma 2.2 of \cite{YLL2014})} Let $G$ be a graph with a subgraph $W_4$. If $G/W_4$ is  $\ZT$-connected, then $G$ is $\ZT$-connected.
    \item  {\rm(Lemmas 2.8 and 2.9 of \cite{YLL2014})} Each of the graphs in Figure \ref{FIG: Z3graph} is $\ZT$-connected.
\end{enumerate} 

\end{lemma}

\begin{lemma}\label{LEM: 4^5_3^4}
If $\pi =(4^5,\, 3^4)$,
then $\pi$ has a $\ZT$-connected realization as depicted in Figure \ref{FIG: W4_4^5,3^4}.

\end{lemma}

\begin{proof}
    Let $G$ denote the realization of $(4^5,\, 3^4)$ in Figure \ref{FIG: W4_4^5,3^4}.
    Let $G'$ be the induced subgraph of $G$ formed by vertices $v_5$, $v_6$, $v_7$, $v_8$ and $v_9$. Obviously, both $G'$ and $G/G'$ are isomorphic to $W_4$. Therefore, by Lemma \ref{LEM: Z3realization} (i) (ii), $G$ is a  $\ZT$-connected graph, which means that $(4^5,\, 3^4) \in \GS(\ZT)$.
\end{proof}

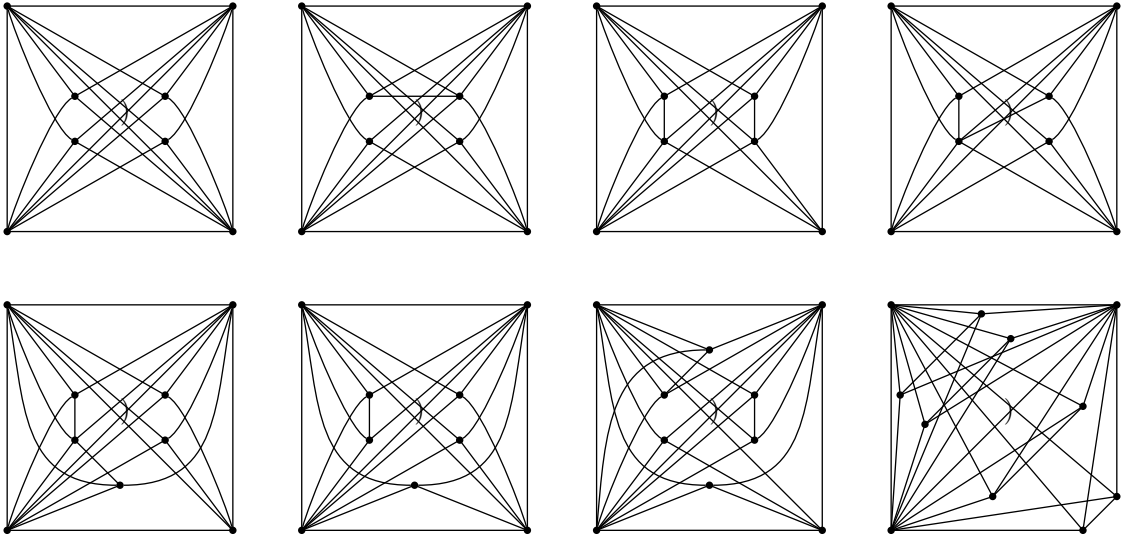
\begin{figure}[ht]
\centering
	\tikzstyle{transition}=[circle,draw=black,fill=black,thick,
	inner sep=0pt, minimum  size=0.7mm]
	\tikzstyle{firstedge}=[line width = 0.5pt,color=black]
	\tikzstyle{secondedge}=[line width = 0.5pt,color=red]
	\tikzstyle{thirdedge}=[line width = 0.5pt,color=green]
	\tikzstyle{fourthedge}=[line width = 0.5pt,color=blue]
 
\begin{tikzpicture}[scale =0.3]
	\node [transition] (v1) at (5,5) {};
	\node [transition] (v2) at (-5,5) {};
	\node [transition] (v3) at (-5,-5) {};
	\node [transition] (v4) at (5,-5) {};
	\node [transition] (v5) at (2,-1) {};
	\node [transition] (v6) at (-2,-1) {};
	\node [transition] (v7) at (-2,1.) {};
	\node [transition] (v8) at (2,1) {};
    \put(2.5,3.3){ $v_1$ }
    \put(-3,3.3){ $v_2$ }
    \put(-3,-3.5){ $v_3$ }
    \put(2.5,-3.5){ $v_4$ }
    \put(1.4,-0.8){ $v_5$ }
    \put(-2,-0.8){ $v_6$ }
    \put(-2,0.6){ $v_7$ }
    \put(1.4,0.6){ $v_8$ }

	\draw [firstedge] plot[smooth, tension=.7] coordinates {(v1) (v2)};
	\draw [firstedge] plot[smooth, tension=.7] coordinates {(v2) (v3)};
	\draw [firstedge] plot[smooth, tension=.7] coordinates {(v3) (v4)};
	\draw [firstedge] plot[smooth, tension=.7] coordinates {(v4) (v1)};
	\draw [firstedge] plot[smooth, tension=.7] coordinates {(v1) (v3)};
	\draw [firstedge] plot[smooth, tension=.7] coordinates {(v2) (v4)};
     \draw[firstedge][out=-105,in=30,looseness = 0.4]  (v1)to (v5);
	%\draw [firstedge] plot[smooth, tension=.7] coordinates {(v5) (v1)};
	\draw [firstedge] plot[smooth, tension=.7] coordinates {(v5) (v2)};
	\draw [firstedge] plot[smooth, tension=.7] coordinates {(v5) (v3)};
	\draw [firstedge] plot[smooth, tension=.7] coordinates {(v5) (v4)};
	\draw [firstedge] plot[smooth, tension=.7] coordinates {(v6) (v1)};
    \draw[firstedge][out=-75,in=150,looseness = 0.4]  (v2)to (v6);
	%\draw [firstedge] plot[smooth, tension=.7] coordinates {(v6) (v2)};
	\draw [firstedge] plot[smooth, tension=.7] coordinates {(v6) (v3)};
	\draw [firstedge] plot[smooth, tension=.7] coordinates {(v6) (v4)};
	\draw [firstedge] plot[smooth, tension=.7] coordinates {(v7) (v1)};
	\draw [firstedge] plot[smooth, tension=.7] coordinates {(v7) (v2)};
    \draw[firstedge][out=75,in=210,looseness = 0.4]  (v3)to (v7)  ;
	%\draw [firstedge] plot[smooth, tension=.7] coordinates {(v7) (v3)};
	\draw [firstedge] plot[smooth, tension=.7] coordinates {(v7) (v4)};

	\draw [firstedge] plot[smooth, tension=.7] coordinates {(v8) (v1)};
	\draw [firstedge] plot[smooth, tension=.7] coordinates {(v8) (v2)};
	\draw [firstedge] plot[smooth, tension=.7] coordinates {(v8) (v3)};
     \draw[firstedge][out=105,in=-30,looseness = 0.4]  (v4)to (v8);
	%\draw [firstedge] plot[smooth, tension=.7] coordinates {(v8) (v4)};
    \put(-1,-3.9){ $(7^4,\, 4^4)$ }
\end{tikzpicture}
\quad \quad
\begin{tikzpicture}[scale =0.3]
		\node [transition] (v1) at (5,5) {};
		\node [transition] (v2) at (-5,5) {};
		\node [transition] (v3) at (-5,-5) {};
		\node [transition] (v4) at (5,-5) {};
		\node [transition] (v5) at (2,-1) {};
		\node [transition] (v6) at (-2,-1) {};
		\node [transition] (v7) at (-2,1.) {};
		\node [transition] (v8) at (2,1) {};
  \put(2.5,3.3){ $v_1$ }
    \put(-3,3.3){ $v_2$ }
    \put(-3,-3.5){ $v_3$ }
    \put(2.5,-3.5){ $v_4$ }
    \put(1.4,-0.8){ $v_5$ }
    \put(-2,-0.8){ $v_6$ }
    \put(-2,0.6){ $v_7$ }
    \put(1.4,0.6){ $v_8$ }

		\draw [firstedge] plot[smooth, tension=.7] coordinates {(v1) (v2)};
		\draw [firstedge] plot[smooth, tension=.7] coordinates {(v2) (v3)};
		\draw [firstedge] plot[smooth, tension=.7] coordinates {(v3) (v4)};
		\draw [firstedge] plot[smooth, tension=.7] coordinates {(v4) (v1)};
		\draw [firstedge] plot[smooth, tension=.7] coordinates {(v1) (v3)};
		\draw [firstedge] plot[smooth, tension=.7] coordinates {(v2) (v4)};
        \draw[firstedge][out=-105,in=30,looseness = 0.4]  (v1)to (v5);
		%\draw [firstedge] plot[smooth, tension=.7] coordinates {(v5) (v1)};
		\draw [firstedge] plot[smooth, tension=.7] coordinates {(v5) (v2)};
		\draw [firstedge] plot[smooth, tension=.7] coordinates {(v5) (v3)};
		\draw [firstedge] plot[smooth, tension=.7] coordinates {(v5) (v4)};
		\draw [firstedge] plot[smooth, tension=.7] coordinates {(v6) (v1)};
        \draw[firstedge][out=-75,in=150,looseness = 0.4]  (v2)to (v6);
		%\draw [firstedge] plot[smooth, tension=.7] coordinates {(v6) (v2)};
		\draw [firstedge] plot[smooth, tension=.7] coordinates {(v6) (v3)};
		\draw [firstedge] plot[smooth, tension=.7] coordinates {(v6) (v4)};
		\draw [firstedge] plot[smooth, tension=.7] coordinates {(v7) (v1)};
		\draw [firstedge] plot[smooth, tension=.7] coordinates {(v7) (v2)};
        \draw[firstedge][out=75,in=210,looseness = 0.4]  (v3)to (v7)  ;
		%\draw [firstedge] plot[smooth, tension=.7] coordinates {(v7) (v3)};
		\draw [firstedge] plot[smooth, tension=.7] coordinates {(v8) (v1)};
		\draw [firstedge] plot[smooth, tension=.7] coordinates {(v8) (v2)};
		\draw [firstedge] plot[smooth, tension=.7] coordinates {(v8) (v3)};
        \draw[firstedge][out=105,in=-30,looseness = 0.4]  (v4)to (v8);
		%\draw [firstedge] plot[smooth, tension=.7] coordinates {(v8) (v4)};
		\draw [firstedge] plot[smooth, tension=.7] coordinates {(v7) (v8)};
  \put(-1.8,-3.9){ $(7^3,\, 6,\, 5,\, 4^3)$}
	\end{tikzpicture}
\quad \quad
	\begin{tikzpicture}[scale =0.3]
		\node [transition] (v1) at (5,5) {};
		\node [transition] (v2) at (-5,5) {};
		\node [transition] (v3) at (-5,-5) {};
		\node [transition] (v4) at (5,-5) {};
		\node [transition] (v5) at (2,-1) {};
		\node [transition] (v6) at (-2,-1) {};
		\node [transition] (v7) at (-2,1.) {};
		\node [transition] (v8) at (2,1) {};
  \put(2.5,3.3){ $v_1$ }
    \put(-3,3.3){ $v_2$ }
    \put(-3,-3.5){ $v_3$ }
    \put(2.5,-3.5){ $v_4$ }
    \put(1.4,-0.8){ $v_5$ }
    \put(-2,-0.8){ $v_6$ }
    \put(-2,0.6){ $v_7$ }
    \put(1.4,0.6){ $v_8$ }

		\draw [firstedge] plot[smooth, tension=.7] coordinates {(v1) (v2)};
		\draw [firstedge] plot[smooth, tension=.7] coordinates {(v2) (v3)};
		\draw [firstedge] plot[smooth, tension=.7] coordinates {(v3) (v4)};
		\draw [firstedge] plot[smooth, tension=.7] coordinates {(v4) (v1)};
		\draw [firstedge] plot[smooth, tension=.7] coordinates {(v1) (v3)};
		\draw [firstedge] plot[smooth, tension=.7] coordinates {(v2) (v4)};
        \draw[firstedge][out=-105,in=30,looseness = 0.4]  (v1)to (v5);
		%\draw [firstedge] plot[smooth, tension=.7] coordinates {(v5) (v1)};
		\draw [firstedge] plot[smooth, tension=.7] coordinates {(v5) (v2)};
		\draw [firstedge] plot[smooth, tension=.7] coordinates {(v5) (v3)};
		\draw [firstedge] plot[smooth, tension=.7] coordinates {(v5) (v4)};
		\draw [firstedge] plot[smooth, tension=.7] coordinates {(v6) (v1)};
        \draw[firstedge][out=-75,in=150,looseness = 0.4]  (v2)to (v6);
		%\draw [firstedge] plot[smooth, tension=.7] coordinates {(v6) (v2)};
		\draw [firstedge] plot[smooth, tension=.7] coordinates {(v6) (v3)};
		\draw [firstedge] plot[smooth, tension=.7] coordinates {(v6) (v4)};
		\draw [firstedge] plot[smooth, tension=.7] coordinates {(v7) (v1)};
		\draw [firstedge] plot[smooth, tension=.7] coordinates {(v7) (v2)};
        \draw[firstedge][out=75,in=210,looseness = 0.4]  (v3)to (v7)  ;
		%\draw [firstedge] plot[smooth, tension=.7] coordinates {(v7) (v3)};
		\draw [firstedge] plot[smooth, tension=.7] coordinates {(v8) (v1)};
		\draw [firstedge] plot[smooth, tension=.7] coordinates {(v8) (v2)};
		\draw [firstedge] plot[smooth, tension=.7] coordinates {(v8) (v3)};
		\draw [firstedge] plot[smooth, tension=.7] coordinates {(v5) (v8)};
		\draw [firstedge] plot[smooth, tension=.7] coordinates {(v6) (v7)};
  \put(-1.6,-3.9){ $(7^3,\, 5^3,\, 4^2)$}
	\end{tikzpicture}
\quad \quad
	\begin{tikzpicture}[scale =0.3]

		\node [transition] (v1) at (5,5) {};
		\node [transition] (v2) at (-5,5) {};
		\node [transition] (v3) at (-5,-5) {};
		\node [transition] (v4) at (5,-5) {};
		\node [transition] (v5) at (2,-1) {};
		\node [transition] (v6) at (-2,-1) {};
		\node [transition] (v7) at (-2,1.) {};
		\node [transition] (v8) at (2,1) {};
    \put(2.5,3.3){ $v_1$ }
    \put(-3,3.3){ $v_2$ }
    \put(-3,-3.5){ $v_3$ }
    \put(2.5,-3.5){ $v_4$ }
    \put(1.4,-0.8){ $v_5$ }
    \put(-2,-0.8){ $v_6$ }
    \put(-2,0.6){ $v_7$ }
    \put(1.4,0.6){ $v_8$ }

		\draw [firstedge] plot[smooth, tension=.7] coordinates {(v1) (v2)};
		\draw [firstedge] plot[smooth, tension=.7] coordinates {(v2) (v3)};
		\draw [firstedge] plot[smooth, tension=.7] coordinates {(v3) (v4)};
		\draw [firstedge] plot[smooth, tension=.7] coordinates {(v4) (v1)};
		\draw [firstedge] plot[smooth, tension=.7] coordinates {(v1) (v3)};
		\draw [firstedge] plot[smooth, tension=.7] coordinates {(v2) (v4)};
        \draw[firstedge][out=-105,in=30,looseness = 0.4]  (v1)to (v5);
		%\draw [firstedge] plot[smooth, tension=.7] coordinates {(v5) (v1)};
		\draw [firstedge] plot[smooth, tension=.7] coordinates {(v5) (v2)};
		\draw [firstedge] plot[smooth, tension=.7] coordinates {(v5) (v3)};
		\draw [firstedge] plot[smooth, tension=.7] coordinates {(v5) (v4)};
		\draw [firstedge] plot[smooth, tension=.7] coordinates {(v6) (v1)};
        \draw[firstedge][out=-75,in=150,looseness = 0.4]  (v2)to (v6);
		%\draw [firstedge] plot[smooth, tension=.7] coordinates {(v6) (v2)};
		\draw [firstedge] plot[smooth, tension=.7] coordinates {(v6) (v3)};
		\draw [firstedge] plot[smooth, tension=.7] coordinates {(v6) (v4)};
		\draw [firstedge] plot[smooth, tension=.7] coordinates {(v7) (v1)};
		\draw [firstedge] plot[smooth, tension=.7] coordinates {(v7) (v2)};
		\draw [firstedge] plot[smooth, tension=.7] coordinates {(v8) (v1)};
		\draw [firstedge] plot[smooth, tension=.7] coordinates {(v8) (v2)};
        \draw[firstedge][out=75,in=210,looseness = 0.4]  (v3)to (v7)  ;
		%\draw [firstedge] plot[smooth, tension=.7] coordinates {(v3) (v7)};
		\draw [firstedge] plot[smooth, tension=.7] coordinates {(v6) (v7)};
		\draw [firstedge] plot[smooth, tension=.7] coordinates {(v6) (v8)};
        \draw[firstedge][out=105,in=-30,looseness = 0.4]  (v4)to (v8);
		%\draw [firstedge] plot[smooth, tension=.7] coordinates {(v4) (v8)};
  \put(-1.6,-3.9){  $(7^2,\, 6^3,\, 4^3)$  }
\end{tikzpicture}
 \\ \quad
 \\ \quad
\\

\begin{tikzpicture}[scale =0.3]
	\tikzstyle{transition}=[circle,draw=black,fill=black,thick,
	inner sep=0pt, minimum  size=0.7mm]
	\tikzstyle{firstedge}=[line width = 0.5pt,color=black]
	\tikzstyle{secondedge}= [densely dashed]
	\node [transition] (v1) at (5,5) {};
	\node [transition] (v2) at (-5,5) {};
	\node [transition] (v3) at (-5,-5) {};
	\node [transition] (v4) at (5,-5) {};
	\node [transition] (v5) at (2,-1) {};
	\node [transition] (v6) at (-2,-1) {};
	\node [transition] (v7) at (-2,1.) {};
	\node [transition] (v8) at (2,1) {};
	\node [transition] (v9) at (-0,-3) {};
 
    \put(2.5,3.3){ $v_1$ }
    \put(-3,3.3){ $v_2$ }
    \put(-3,-3.5){ $v_3$ }
    \put(2.5,-3.5){ $v_4$ }
    \put(1.4,-0.8){ $v_5$ }
    \put(-2,-0.8){ $v_6$ }
    \put(-2,0.6){ $v_7$ }
    \put(1.4,0.6){ $v_8$ }
    \put(-0.3,-2.4){ $v_9$ }

	\draw [firstedge] plot[smooth, tension=.7] coordinates {(v1) (v2)};
	\draw [firstedge] plot[smooth, tension=.7] coordinates {(v1) (v3)};
	\draw [firstedge] plot[smooth, tension=.7] coordinates {(v1) (v4)};
	\draw[firstedge][out=-105,in=30,looseness = 0.4]  (v1)to (v5);
	\draw [firstedge] plot[smooth, tension=.7] coordinates {(v1) (v6)};
	\draw [firstedge] plot[smooth, tension=.7] coordinates {(v1) (v7)};
	\draw [firstedge] plot[smooth, tension=.7] coordinates {(v1) (v8)};
    \draw[firstedge][out=-100,in=00,looseness = 1.2]  (v1)to (v9)  ;
	%\draw [firstedge] plot[smooth, tension=.7] coordinates {(v1) (v9)};
	\draw [firstedge] plot[smooth, tension=.7] coordinates {(v2) (v3)};
	\draw [firstedge] plot[smooth, tension=.7] coordinates {(v2) (v4)};
	\draw [firstedge] plot[smooth, tension=.7] coordinates {(v2) (v5)};
	\draw[firstedge][out=-75,in=150,looseness = 0.4]  (v2)to (v6);
	\draw [firstedge] plot[smooth, tension=.7] coordinates {(v2) (v7)};
	\draw [firstedge] plot[smooth, tension=.7] coordinates {(v2) (v8)};
    \draw[firstedge][out=-80,in=180,looseness = 1.2]  (v2)to (v9)  ;
	%\draw [firstedge] plot[smooth, tension=.7] coordinates {(v2) (v9)};
	\draw [firstedge] plot[smooth, tension=.7] coordinates {(v3) (v4)};
	\draw [firstedge] plot[smooth, tension=.7] coordinates {(v3) (v5)};
	\draw [firstedge] plot[smooth, tension=.7] coordinates {(v3) (v6)};
	\draw[firstedge][out=75,in=210,looseness = 0.4]  (v3)to (v7)  ;
	\draw [firstedge] plot[smooth, tension=.7] coordinates {(v3) (v8)};
	\draw [firstedge] plot[smooth, tension=.7] coordinates {(v3) (v9)};
	\draw [firstedge] plot[smooth, tension=.7] coordinates {(v4) (v5)};
	\draw[firstedge][out=105,in=-30,looseness = 0.4]  (v4)to (v8);
	\draw [firstedge] plot[smooth, tension=.7] coordinates {(v6) (v7)};
	\draw [firstedge] plot[smooth, tension=.7] coordinates {(v6) (v9)};
 \put(-1.6,-3.9){  $(8^{3},\, 5^2,\, 4^4)$ }
\end{tikzpicture}
\quad \quad
\begin{tikzpicture}[scale = 0.3]
	\tikzstyle{transition}=[circle,draw=black,fill=black,thick,
	inner sep=0pt, minimum  size=0.7mm]
	\tikzstyle{firstedge}=[line width = 0.5pt,color=black]
	\tikzstyle{secondedge}= [densely dashed]

	\node [transition] (v1) at (5,5) {};
	\node [transition] (v2) at (-5,5) {};
	\node [transition] (v3) at (-5,-5) {};
	\node [transition] (v4) at (5,-5) {};
	\node [transition] (v5) at (2,-1) {};
	\node [transition] (v6) at (-2,-1) {};
	\node [transition] (v7) at (-2,1.) {};
	\node [transition] (v8) at (2,1) {};
	\node [transition] (v9) at (-0,-3) {};
    \put(2.5,3.3){ $v_1$ }
    \put(-3,3.3){ $v_2$ }
    \put(-3,-3.5){ $v_3$ }
    \put(2.5,-3.5){ $v_4$ }
    \put(1.4,-0.8){ $v_5$ }
    \put(-2,-0.8){ $v_6$ }
    \put(-2,0.6){ $v_7$ }
    \put(1.4,0.6){ $v_8$ }
    \put(-0.3,-2.4){ $v_9$ }

	\draw [firstedge] plot[smooth, tension=.7] coordinates {(v1) (v2)};
	\draw [firstedge] plot[smooth, tension=.7] coordinates {(v1) (v3)};
	\draw [firstedge] plot[smooth, tension=.7] coordinates {(v1) (v4)};
    \draw[firstedge][out=-105,in=30,looseness = 0.4]  (v1)to (v5);
	%\draw [firstedge] plot[smooth, tension=.7] coordinates {(v1) (v5)};
	\draw [firstedge] plot[smooth, tension=.7] coordinates {(v1) (v6)};
	\draw [firstedge] plot[smooth, tension=.7] coordinates {(v1) (v7)};
	\draw [firstedge] plot[smooth, tension=.7] coordinates {(v1) (v8)};
    \draw[firstedge][out=-100,in=00,looseness = 1.2]  (v1)to (v9)  ;
	%\draw [firstedge] plot[smooth, tension=.7] coordinates {(v1) (v9)};
	\draw [firstedge] plot[smooth, tension=.7] coordinates {(v2) (v3)};
	\draw [firstedge] plot[smooth, tension=.7] coordinates {(v2) (v4)};
	\draw [firstedge] plot[smooth, tension=.7] coordinates {(v2) (v5)};
    \draw[firstedge][out=-75,in=150,looseness = 0.4]  (v2)to (v6);
	%\draw [firstedge] plot[smooth, tension=.7] coordinates {(v2) (v6)};
	\draw [firstedge] plot[smooth, tension=.7] coordinates {(v2) (v7)};
	\draw [firstedge] plot[smooth, tension=.7] coordinates {(v2) (v8)};
    \draw[firstedge][out=-80,in=180,looseness = 1.2]  (v2)to (v9)  ;
	%\draw [firstedge] plot[smooth, tension=.7] coordinates {(v2) (v9)};
	\draw [firstedge] plot[smooth, tension=.7] coordinates {(v3) (v4)};
	\draw [firstedge] plot[smooth, tension=.7] coordinates {(v3) (v5)};
	\draw [firstedge] plot[smooth, tension=.7] coordinates {(v3) (v6)};
	\draw[firstedge][out=75,in=210,looseness = 0.4]  (v3)to (v7)  ;
    %\draw [firstedge] plot[smooth, tension=.7] coordinates {(v3) (v7)};
	\draw [firstedge] plot[smooth, tension=.7] coordinates {(v3) (v8)};
	\draw [firstedge] plot[smooth, tension=.7] coordinates {(v3) (v9)};
	\draw [firstedge] plot[smooth, tension=.7] coordinates {(v4) (v5)};
    \draw[firstedge][out=105,in=-30,looseness = 0.4]  (v4)to (v8);
	%\draw [firstedge] plot[smooth, tension=.7] coordinates {(v4) (v8)};
	\draw [firstedge] plot[smooth, tension=.7] coordinates {(v4) (v9)};
	\draw [firstedge] plot[smooth, tension=.7] coordinates {(v6) (v7)};
 \put(-1.6,-3.9){ $(8^3,\, 6,\, 4^5)$ }
\end{tikzpicture}
\quad \quad
\begin{tikzpicture}[scale = 0.3]
	\tikzstyle{transition}=[circle,draw=black,fill=black,thick,
	inner sep=0pt, minimum  size=0.7mm]
	\tikzstyle{firstedge}=[line width = 0.5pt,color=black]
	\tikzstyle{secondedge}= [densely dashed]

	\node [transition] (v1) at (5,5) {};
	\node [transition] (v2) at (-5,5) {};
	\node [transition] (v3) at (-5,-5) {};
	\node [transition] (v4) at (5,-5) {};
	\node [transition] (v5) at (2,-1) {};
	\node [transition] (v6) at (-2,-1) {};
	\node [transition] (v7) at (-2,1.) {};
	\node [transition] (v8) at (2,1) {};
	\node [transition] (v9) at (-0,-3) {};
	\node [transition] (v10) at (-0,3) {};
     \put(2.5,3.3){ $v_1$ }
    \put(-3,3.3){ $v_2$ }
    \put(-3,-3.5){ $v_3$ }
    \put(2.5,-3.5){ $v_4$ }
    \put(1.4,-0.8){ $v_5$ }
    \put(-2,-0.8){ $v_6$ }
    \put(-2,0.6){ $v_7$ }
    \put(1.4,0.6){ $v_8$ }
    \put(-0.3,-2.4){ $v_9$ }
     \put(-0.3,2.3){ $v_{10}$ }

	\draw [firstedge] plot[smooth, tension=.7] coordinates {(v1) (v2)};
	\draw [firstedge] plot[smooth, tension=.7] coordinates {(v1) (v3)};
	\draw [firstedge] plot[smooth, tension=.7] coordinates {(v1) (v4)};
    \draw[firstedge][out=-105,in=30,looseness = 0.4]  (v1)to (v5);
	%\draw [firstedge] plot[smooth, tension=.7] coordinates {(v1) (v5)};
	\draw [firstedge] plot[smooth, tension=.7] coordinates {(v1) (v6)};
	\draw [firstedge] plot[smooth, tension=.7] coordinates {(v1) (v7)};
	\draw [firstedge] plot[smooth, tension=.7] coordinates {(v1) (v8)};
    \draw[firstedge][out=-100,in=00,looseness = 1.2]  (v1)to (v9)  ;
	%\draw [firstedge] plot[smooth, tension=.7] coordinates {(v1) (v9)};
	\draw [firstedge] plot[smooth, tension=.7] coordinates {(v1) (v10)};
	\draw [firstedge] plot[smooth, tension=.7] coordinates {(v2) (v3)};
	\draw [firstedge] plot[smooth, tension=.7] coordinates {(v2) (v4)};
	\draw [firstedge] plot[smooth, tension=.7] coordinates {(v2) (v5)};
    \draw[firstedge][out=-75,in=150,looseness = 0.4]  (v2)to (v6);
	%\draw [firstedge] plot[smooth, tension=.7] coordinates {(v2) (v6)};
	\draw [firstedge] plot[smooth, tension=.7] coordinates {(v2) (v7)};
	\draw [firstedge] plot[smooth, tension=.7] coordinates {(v2) (v8)};
    \draw[firstedge][out=-80,in=180,looseness = 1.2]  (v2)to (v9);
	%\draw [firstedge] plot[smooth, tension=.7] coordinates {(v2) (v9)};
	\draw [firstedge] plot[smooth, tension=.7] coordinates {(v2) (v10)};
	\draw [firstedge] plot[smooth, tension=.7] coordinates {(v3) (v4)};
	\draw [firstedge] plot[smooth, tension=.7] coordinates {(v3) (v5)};
	\draw [firstedge] plot[smooth, tension=.7] coordinates {(v3) (v6)};
    \draw[firstedge][out=75,in=210,looseness = 0.4]  (v3)to (v7);
	%\draw [firstedge] plot[smooth, tension=.7] coordinates {(v3) (v7)};
	\draw [firstedge] plot[smooth, tension=.7] coordinates {(v3) (v8)};
	\draw [firstedge] plot[smooth, tension=.7] coordinates {(v3) (v9)};
    \draw[firstedge][out=85,in=180,looseness = 1.2]  (v3)to (v10);
	%\draw [firstedge] plot[smooth, tension=.7] coordinates {(v3) (v10)};
	\draw [firstedge] plot[smooth, tension=.7] coordinates {(v4) (v6)};
	\draw [firstedge] plot[smooth, tension=.7] coordinates {(v4) (v9)};
	\draw [firstedge] plot[smooth, tension=.7] coordinates {(v5) (v8)};
	\draw [firstedge] plot[smooth, tension=.7] coordinates {(v7) (v10)};
 \put(-1.5,-3.9){$(9^3,\, 5,\, 4^6)$ }
\end{tikzpicture}
\quad \quad
\begin{tikzpicture}[scale =0.3]
	\tikzstyle{transition}=[circle,draw=black,fill=black,thick,
	inner sep=0pt, minimum  size=0.7mm]
	\tikzstyle{firstedge}=[line width = 0.5pt,color=black]
	\tikzstyle{secondedge}= [densely dashed]-
	\node [transition] (v1) at (5,5) {};
	\node [transition] (v2) at (-5,5) {};
	\node [transition] (v3) at (-5,-5) {};
 
	\node [transition] (v4) at (-4.6,1) {};
	\node [transition] (v5) at (-1,4.6) {};
	\node [transition] (v6) at (3.5,0.5) {};
	\node [transition] (v7) at (5,-3.5) {};
	\node [transition] (v8) at (-0.5,-3.5) {};
	\node [transition] (v9) at (0.3,3.5) {};%9
	\node [transition] (v10) at (3.5,-5) {};
	\node [transition] (v11) at (-3.5,-0.3) {};%8

    \put(2.5,3.3){ $v_1$ }
    \put(-3,3.3){ $v_2$ }
    \put(-3,-3.5){ $v_3$ }
    \put(1.8,-3.5){ $v_4$ }
    \put(3.1,-2.3){ $v_5$ }
    \put(-0.1,-2.3){ $v_6$ }
    \put(1.9,-0.2){ $v_7$ }
    \put(-2.8,-0.2){ $v_8$ }
    \put(-0.3,2.35){ $v_9$ }
     \put(-3.93,0.6){ $v_{10}$ }
    \put(-1.1,3.3){ $v_{11}$ }

	\draw [firstedge] plot[smooth, tension=.7] coordinates {(v1) (v2)};
	\draw [firstedge] plot[smooth, tension=.7] coordinates {(v1) (v3)};
	\draw [firstedge] plot[smooth, tension=.7] coordinates {(v1) (v4)};
	\draw [firstedge] plot[smooth, tension=.7] coordinates {(v1) (v5)};
	\draw [firstedge] plot[smooth, tension=.7] coordinates {(v1) (v6)};
	\draw [firstedge] plot[smooth, tension=.7] coordinates {(v1) (v7)};
	\draw [firstedge] plot[smooth, tension=.7] coordinates {(v1) (v8)};
	\draw [firstedge] plot[smooth, tension=.7] coordinates {(v1) (v9)};
	\draw [firstedge] plot[smooth, tension=.7] coordinates {(v1) (v10)};
	\draw [firstedge] plot[smooth, tension=.7] coordinates {(v1) (v11)};
	\draw [firstedge] plot[smooth, tension=.7] coordinates {(v2) (v3)};
	\draw [firstedge] plot[smooth, tension=.7] coordinates {(v2) (v4)};
	\draw [firstedge] plot[smooth, tension=.7] coordinates {(v2) (v5)};
	\draw [firstedge] plot[smooth, tension=.7] coordinates {(v2) (v6)};
	\draw [firstedge] plot[smooth, tension=.7] coordinates {(v2) (v7)};
	\draw [firstedge] plot[smooth, tension=.7] coordinates {(v2) (v8)};
	\draw [firstedge] plot[smooth, tension=.7] coordinates {(v2) (v9)};
	\draw [firstedge] plot[smooth, tension=.7] coordinates {(v2) (v10)};
	\draw [firstedge] plot[smooth, tension=.7] coordinates {(v2) (v11)};
	\draw [firstedge] plot[smooth, tension=.7] coordinates {(v3) (v4)};
	\draw [firstedge] plot[smooth, tension=.7] coordinates {(v3) (v5)};
	\draw [firstedge] plot[smooth, tension=.7] coordinates {(v3) (v6)};
	\draw [firstedge] plot[smooth, tension=.7] coordinates {(v3) (v7)};
	\draw [firstedge] plot[smooth, tension=.7] coordinates {(v3) (v8)};
	\draw [firstedge] plot[smooth, tension=.7] coordinates {(v3) (v9)};
	\draw [firstedge] plot[smooth, tension=.7] coordinates {(v3) (v10)};
	\draw [firstedge] plot[smooth, tension=.7] coordinates {(v3) (v11)};
	\draw [firstedge] plot[smooth, tension=.7] coordinates {(v4) (v5)};
	\draw [firstedge] plot[smooth, tension=.7] coordinates {(v6) (v8)};
	\draw [firstedge] plot[smooth, tension=.7] coordinates {(v11) (v9)};
	\draw [firstedge] plot[smooth, tension=.7] coordinates {(v7) (v10)};
 \put(-1.5,-3.9){ $(10^3,\, 4^8)$}
\end{tikzpicture}
\\ \quad
	\caption{ The \ST-realizations of $(7^4,\, 4^4)$, $(7^3,\, 6,\, 5,\, 4^3)$, $(7^3,\, 5^3,\, 4^2)$, $(7^2,\, 6^3,\, 4^3)$, $(8^{3},\, 5^2,\, 4^4)$, $(8^3,\, 6,\, 4^5)$, $(9^3,\, 5,\, 4^6)$ and $(10^3,\, 4^8)$.}\label{FIG: lift}
\end{figure}

Next, we provide some $\ST$-connected realizations for certain specific graphic sequences.
\begin{lemma}\label{LEM: S3realizationbylifting}

If $$\pi \in \{
(7^4,\, 4^4),\,   (7^3,\, 6,\, 5,\, 4^3),\,   (7^3,\, 5^3,\, 4^2),\,  (7^2,\, 6^3,\, 4^3),\,  (8^{3},\, 5^2,\, 4^4),\,  (8^3,\, 6,\, 4^5),\,  (9^3,\, 5,\, 4^6),\,  (10^3,\, 4^8)  \},\, $$
then $\pi$ has an $\ST$-connected realization.
\end{lemma}
\begin{proof}
    Figure \ref{FIG: lift} shows some realizations of the above mentioned degree sequences. %Here, we only demonstrate in detail that the first graph in Figure \ref{FIG: lift} is \ST-connected, and the proofs for the remaining graphs are similar and are listed briefly in later paragraphs.
    %Let $G$ be the first graph in Figure \ref{FIG: lift}, which is the realization of $(7^4,\, 4^4)$. In order to apply the lifting to each vertex of degree 4, we proceed as follows: Firstly, we remove vertex $v_5$ together with all incident edges and add a new edge $v_3v_4$. Consequently, the resulting graph is $G_{[v_5,\, v_3v_4]}$, denoted by $G'$. Secondly, we delete vertex $v_6$ along with  all incident edges and add a new edge $v_2v_3$. As a result, the resulting graph is $G'_{[v_6,\, v_2v_3]}$, denoted by $G''$. Thirdly, we remove vertex $v_7$ together with its incident edges and add a new edge $v_1v_2$. Subsequently, the resulting graph is $G''_{[v_7,\, v_1v_2]}$, denoted by $G'''$. Finally, we delete vertex $v_8$ along with its incident edges and add a new edge $v_1v_4$. Lastly, the resulting graph is $G'''_{[v_8,\, v_1v_4]}$. Note that the resulting graph $G'''_{[v_8,\, v_1v_4]}$ is ultimately isomorphic to $K_4^*$, which is \ST-connected by Lemma \ref{LEM: S3graph} (iii). Then, by applying Lemma \ref{LEM: liftinggraph} recursively, the original graph $G$ is an $\ST$-connected realization of $(7^4,\, 4^4)$.
    %Note that, the lifting transformation from the realization of $(7^4,\,4^4)$ to $K_4^*$ above can be concisely expressed as $ G \rightarrow G_{[v_5,\, v_3v_4]} \rightarrow  [v_6,\, v_2v_3] \rightarrow [v_7,\, v_1v_2] \rightarrow [v_8,\, v_1v_4]$.  Next, this form is used to represent the lifting process for the other graphs.
    
Let $G$ be a realization of $(7^4,\,4^4)$ in Figure \ref{FIG: lift}. The lifting process is as follows: $ G \rightarrow G_{[v_5,\, v_3v_4]} \rightarrow  G_{[v_6,\, v_2v_3]} \rightarrow G_{[v_7,\, v_1v_2]} \rightarrow G_{[v_8,\, v_1v_4]}$.  
Lastly, the resulting graph is isomorphic to $K_4^*$.
By Lemmas \ref{LEM: S3graph} (iii) and \ref{LEM: liftinggraph}, $G$ is an \ST-realization of $(7^4,\,4^4)$.

Let $G$ be a realization of $(7^3,\, 6,\, 5,\, 4^3)$ in Figure \ref{FIG: lift}. The lifting process is as follows: $ G \rightarrow G_{[v_5,\, v_1v_4]} \rightarrow  G_{[v_6,\, v_3v_4]} \rightarrow G_{[v_7,\, v_2v_3]} \rightarrow G_{[v_8,\, v_1v_2]}$.  
Consequently, the resulting graph is isomorphic to $K_4^*$.
Based on Lemmas \ref{LEM: S3graph} (iii) and \ref{LEM: liftinggraph}, $G$ is an \ST-realization of $(7^3,\, 6,\, 5,\, 4^3)$. 

Let $G$ be a realization of $(7^3,\, 5^3,\, 4^2)$ in Figure \ref{FIG: lift}. The lifting process is as follows: $ G \rightarrow G_{[v_7,\, v_1v_2]} \rightarrow  G_{[v_8,\, v_2v_3]} \rightarrow G_{[v_5,\, v_3v_4]} \rightarrow G_{[v_6,\, v_1v_4]}$. 
Finally, the resulting graph is isomorphic to $K_4^*$.
According to Lemmas \ref{LEM: S3graph} (iii) and \ref{LEM: liftinggraph}, $G$ is an \ST-realization of $(7^3,\, 5^3,\, 4^2)$. 

Let $G$ be a realization of $(7^2,\, 6^3,\, 4^3)$ in Figure \ref{FIG: lift}. The lifting process is  as follows:  
 $ G \rightarrow G_{[v_5,\, v_1v_4]} \rightarrow G_{[v_7,\, v_2v_3]} \rightarrow G_{[v_8,\, v_1v_2]} \rightarrow G_{[v_6,\, v_3v_4]}$.
Lastly, the resulting graph is isomorphic to $K_4^*$.
By Lemmas \ref{LEM: S3graph} (iii) and \ref{LEM: liftinggraph}, $G$ is an \ST-realization of $(7^2,\, 6^3,\, 4^3)$.

Let $G$ be a realization of $(8^{3},\, 5^2,\, 4^4)$ in Figure \ref{FIG: lift}. The lifting process is as follows: 
 $ G \rightarrow G_{[v_9,\, v_3v_6]} \rightarrow  G_{[v_5,\, v_3v_4]} \rightarrow G_{[v_8,\, v_1v_4]} \rightarrow G_{[v_7,\, v_1v_2]} \rightarrow G_{[v_6,\, v_2v_3]}$.
Consequently, the resulting graph is isomorphic to $K_4^*$.
According to Lemmas \ref{LEM: S3graph} (iii) and \ref{LEM: liftinggraph}, $G$ is an \ST-realization of $(8^{3},\, 5^2,\, 4^4)$.

Let $G$ be a realization of $(8^3,\, 6,\, 4^5)$ in Figure \ref{FIG: lift}. The lifting process is as follows: 
 $ G \rightarrow G_{[v_6,\, v_3v_7]} \rightarrow  G_{[v_9,\, v_3v_4]} \rightarrow G_{[v_5,\, v_1v_4]} \rightarrow G_{[v_8,\, v_1v_2]} \rightarrow G_{[v_7,\, v_2v_3]}$.
Lastly, the resulting graph is isomorphic to $K_4^*$.
Based on Lemmas \ref{LEM: S3graph} (iii) and \ref{LEM: liftinggraph}, $G$ is an \ST-realization of $(8^3,\, 6,\, 4^5)$.

Let $G$ be a realization of $(9^3,\, 5,\, 4^6)$ in Figure \ref{FIG: lift}. The lifting process is  as follows:  
 $ G \rightarrow G_{[v_{10},\, v_2v_7]} \rightarrow  G_{[v_8,\, v_1v_5]} \rightarrow G_{[v_9,\, v_3v_4]} \rightarrow G_{[v_6,\, v_1v_4]} \rightarrow G_{[v_5,\, v_1v_2]} \rightarrow G_{[v_7,\, v_2v_3]}$.
Finally, the resulting graph is isomorphic to $K_4^*$.
By Lemmas \ref{LEM: S3graph} (iii) and \ref{LEM: liftinggraph}, $G$ is an \ST-realization of $(9^3,\, 5,\, 4^6)$.

Let $G$ be a realization of $(10^3,\, 4^8)$ in Figure \ref{FIG: lift}. The lifting process is as follows: 
 $ G \rightarrow G_{[v_{4},\, v_1v_5]} \rightarrow  G_{[v_5,\, v_2v_3]} \rightarrow G_{[v_6,\, v_1v_7]} \rightarrow G_{[v_7,\, v_2v_3]} \rightarrow G_{[v_8,\, v_2v_9]} \rightarrow G_{[v_9,\, v_1v_2]} \rightarrow G_{[v_{10},\, v_2v_{11}]} \rightarrow  G_{[v_{11},\, v_1v_2]}$.
Consequently, the resulting graph is isomorphic to $K_{(1,3,3)}$.
According to Lemmas \ref{LEM: S3graph} (iii) and \ref{LEM: liftinggraph}, $G$ is an \ST-realization of $(10^3,\, 4^8)$.
\end{proof}

%Mr. Li provides a definition of $\pi_{-2}$ in the fourth section
% We define $\pi_{-2}  = (d''_1,\, \ldots,\, d''_{n})$ as the sequence obtained by subtracting 2 from each element in $\pi = (d_1,\, \ldots,\, d_{n})$. In other words, $\pi_{-2}$ is the sequence $(d_1-2,\,  \dots,\, d_{n}-2)$. If $\pi_{-2} $ has a $\ZT$-realization and its complement has a Hamiltonian cycle, it is straightforward to construct an $\ST$-realization for $\pi$.

\begin{lemma}\label{LEM: S3realizationbyaddingHC}
If $$\pi \in \{   (6^3,\, 5^4),\,  (6^{4},\, 5^4),\, (7,\, 6^{2},\, 5^5),\, (7^2,\, 5^6),\,  (6^{5},\, 5^4)
  \},$$
then $\pi$ has an $\ST$-connected realization.
\end{lemma}

    \begin{proof}
    In Figures \ref{FIG: W4_4^5,3^4} and \ref{FIG: Z3+HC}, we characterize the realizations of these sequences. Let $G$ be one of the realizations.
     According to Lemmas \ref{LEM: Z3realization} (iii) and \ref{LEM: 4^5_3^4}, each graph is composed of a $\ZT$-connected graph and a Hamiltonian cycle. Therefore, based on Lemma \ref{LEM: Z_3+H=S3}, $G$ belongs to $\SST$.    
     \end{proof}

\begin{figure}[ht]
    \centering
    \begin{tikzpicture}[scale =0.3]
	\tikzstyle{transition}=[circle,draw=black,fill=black,thick,
	inner sep=0pt, minimum  size=0.7mm]
	\tikzstyle{firstedge}=[line width = 0.5pt,color=black]
	\tikzstyle{secondedge}= [densely dashed]

	\node [transition] (v1) at (5,5) {};
	\node [transition] (v2) at (-5,5) {};
	\node [transition] (v3) at (-5,-5) {};
	\node [transition] (v4) at (5,-5) {};
	\node [transition] (v5) at (2,-1) {};
	\node [transition] (v6) at (-2,-1) {};
	\node [transition] (v7) at (-2,1.) {};

	\draw [firstedge] plot[smooth, tension=.7] coordinates {(v1) (v2)};
	\draw [firstedge] plot[smooth, tension=.7] coordinates {(v1) (v4)};
	\draw [firstedge] plot[smooth, tension=.7] coordinates {(v1) (v7)};
	\draw [firstedge] plot[smooth, tension=.7] coordinates {(v2) (v3)};
	\draw [firstedge] plot[smooth, tension=.7] coordinates {(v2) (v7)};
	\draw [firstedge] plot[smooth, tension=.7] coordinates {(v3) (v4)};
	\draw [firstedge] plot[smooth, tension=.7] coordinates {(v3) (v5)};
	\draw [firstedge] plot[smooth, tension=.7] coordinates {(v3) (v6)};
	\draw [firstedge] plot[smooth, tension=.7] coordinates {(v4) (v5)};
	\draw [firstedge] plot[smooth, tension=.7] coordinates {(v4) (v6)};
	\draw [firstedge] plot[smooth, tension=.7] coordinates {(v5) (v7)};
	\draw [firstedge] plot[smooth, tension=.7] coordinates {(v6) (v7)};

	%\draw [secondedge] plot[smooth, tension=.7] coordinates {(v1) (v3)};
    \draw[secondedge][out=-100,in=10,looseness = 1.2]  (v1)to (v3);
 
	\draw [secondedge][out=80,in=190,looseness = 1.]  (v3)to (v7);
	\draw [secondedge] [out=0,in=100,looseness = 1.2]  (v7)to (v4);
	\draw [secondedge] [out=170,in=-80,looseness = 1.2]  (v4)to (v2);
	\draw [secondedge] [out=-10,in=100,looseness = 1.2]  (v2)to (v5);
	\draw [secondedge] [out=190,in=-10,looseness = 1.2]  (v5)to (v6);
	\draw [secondedge] [out=60,in=-125,looseness = 1.2]  (v6)to (v1);
    \put(-1,-3.9){ $(6^3,\, 5^4)$,}
\end{tikzpicture}
\quad \quad
    \begin{tikzpicture}[scale =0.3]
		\tikzstyle{transition}=[circle,draw=black,fill=black,thick,
		inner sep=0pt, minimum  size=0.7mm]
		\tikzstyle{firstedge}=[line width = 0.5pt,color=black]
		\tikzstyle{secondedge}= [densely dashed]

		\node [transition] (v1) at (5,5) {};
		\node [transition] (v2) at (-5,5) {};
		\node [transition] (v3) at (-5,-5) {};
		\node [transition] (v4) at (5,-5) {};
		\node [transition] (v5) at (2,-1) {};
		\node [transition] (v6) at (-2,-1) {};
		\node [transition] (v7) at (-2,1.) {};
		\node [transition] (v8) at (2,1) {};

		\draw [firstedge] plot[smooth, tension=.7] coordinates {(v1) (v2)};
		\draw [firstedge] plot[smooth, tension=.7] coordinates {(v1) (v4)};
		\draw [firstedge] plot[smooth, tension=.7] coordinates {(v1) (v5)};
		\draw [firstedge] plot[smooth, tension=.7] coordinates {(v1) (v8)};
		\draw [firstedge] plot[smooth, tension=.7] coordinates {(v2) (v3)};
		\draw [firstedge] plot[smooth, tension=.7] coordinates {(v2) (v4)};
		\draw [firstedge] plot[smooth, tension=.7] coordinates {(v2) (v8)};
		\draw [firstedge] plot[smooth, tension=.7] coordinates {(v3) (v4)};
		\draw [firstedge] plot[smooth, tension=.7] coordinates {(v3) (v6)};
		\draw [firstedge] plot[smooth, tension=.7] coordinates {(v3) (v7)};
		\draw [firstedge] plot[smooth, tension=.7] coordinates {(v4) (v5)};
		\draw [firstedge] plot[smooth, tension=.7] coordinates {(v5) (v6)};
		\draw [firstedge] plot[smooth, tension=.7] coordinates {(v6) (v7)};
		\draw [firstedge] plot[smooth, tension=.7] coordinates {(v7) (v8)};

		\draw [secondedge] [out=190,in=80,looseness = 1]  (v1)to (v6);
		\draw [secondedge][out=-60,in=160,looseness = 1]  (v6)to (v4);
		\draw [secondedge][out=100,in=-30,looseness = 1]  (v4)to (v8);
		\draw [secondedge] [out=-160,in=25,looseness = 1]  (v8)to (v3);
		\draw [secondedge] [out=10,in=-120,looseness = 1]  (v3)to (v5);
		\draw [secondedge] [out=110,in=-10,looseness = 1]  (v5)to (v2);
		\draw [secondedge] [out=-80,in=160,looseness = 1]  (v2)to (v7);
		\draw [secondedge] [out=10,in=-150,looseness = 1]  (v7)to (v1);
\put(-1,-3.9){  $(6^4,\,  5^4)$}
	\end{tikzpicture}
 \quad \quad
\begin{tikzpicture}[scale =0.3]
	\tikzstyle{transition}=[circle,draw=black,fill=black,thick,
	inner sep=0pt, minimum  size=0.7mm]
	\tikzstyle{firstedge}=[line width = 0.5pt,color=black]
	\tikzstyle{secondedge}= [densely dashed]

	\node [transition] (v1) at (5,5) {};
	\node [transition] (v2) at (-5,5) {};
	\node [transition] (v3) at (-5,-5) {};
	\node [transition] (v4) at (5,-5) {};
	\node [transition] (v5) at (2,-1) {};
	\node [transition] (v6) at (-2,-1) {};
	\node [transition] (v7) at (-2,1.) {};
	\node [transition] (v8) at (2,1) {};

	\draw [firstedge] plot[smooth, tension=.7] coordinates {(v1) (v2)};
	\draw [firstedge] plot[smooth, tension=.7] coordinates {(v1) (v4)};
	\draw [firstedge] plot[smooth, tension=.7] coordinates {(v1) (v6)};
	\draw [firstedge] plot[smooth, tension=.7] coordinates {(v1) (v7)};
	\draw [firstedge] plot[smooth, tension=.7] coordinates {(v1) (v8)};
	\draw [firstedge] plot[smooth, tension=.7] coordinates {(v2) (v3)};
	\draw [firstedge] plot[smooth, tension=.7] coordinates {(v2) (v4)};
	\draw [firstedge] plot[smooth, tension=.7] coordinates {(v2) (v6)};
	\draw [firstedge] plot[smooth, tension=.7] coordinates {(v3) (v4)};
	\draw [firstedge] plot[smooth, tension=.7] coordinates {(v3) (v5)};
	\draw [firstedge] plot[smooth, tension=.7] coordinates {(v4) (v5)};
	\draw [firstedge] plot[smooth, tension=.7] coordinates {(v5) (v8)};
	\draw [firstedge] plot[smooth, tension=.7] coordinates {(v6) (v7)};
	\draw [firstedge] plot[smooth, tension=.7] coordinates {(v7) (v8)};

	\draw [secondedge] [out=190,in=80,looseness = 1.3]  (v1)to (v3);
	\draw [secondedge] [out=70,in=-110,looseness = 1]  (v3)to (v8);
	\draw [secondedge] [out=-135,in=25,looseness = 1]  (v8)to (v6);
	\draw [secondedge] [out=-70,in=160,looseness = 1]  (v6)to (v4);
	\draw [secondedge] [out=150,in=-50,looseness = 1]  (v4)to (v7);
	\draw [secondedge] [out=135,in=-55,looseness = 1]  (v7)to (v2);
	\draw [secondedge] [out=-20,in=135,looseness = 1]  (v2)to (v5);
	\draw [secondedge] [out=30,in=-105,looseness = 1]  (v5)to (v1);
 \put(-1.5,-3.9){ $(7,\,  6^2,\,  5^5)$}
\end{tikzpicture}
\quad \quad
	\begin{tikzpicture}[scale = 0.3]
 \tikzstyle{transition}=[circle,draw=black,fill=black,thick,
	inner sep=0pt, minimum  size=0.7mm]
	\tikzstyle{firstedge}=[line width = 0.5pt,color=black]
	\tikzstyle{secondedge}= [densely dashed]
		\node [transition] (v1) at (5,5) {};
		\node [transition] (v2) at (-5,5) {};
		\node [transition] (v3) at (-5,-5) {};
		\node [transition] (v4) at (5,-5) {};
		\node [transition] (v5) at (2,-1) {};
		\node [transition] (v6) at (-2,-1) {};
		\node [transition] (v7) at (-2,1.) {};
		\node [transition] (v8) at (2,1) {};
		\draw [firstedge] plot[smooth, tension=.7] coordinates {(v1) (v2)};
		\draw [firstedge] plot[smooth, tension=.7] coordinates {(v1) (v4)};
		\draw [firstedge] plot[smooth, tension=.7] coordinates {(v1) (v7)};
		\draw [firstedge] plot[smooth, tension=.7] coordinates {(v2) (v3)};
		\draw [firstedge] plot[smooth, tension=.7] coordinates {(v2) (v7)};
		\draw [firstedge] plot[smooth, tension=.7] coordinates {(v3) (v4)};
		\draw [firstedge] plot[smooth, tension=.7] coordinates {(v3) (v7)};
		\draw [firstedge] plot[smooth, tension=.7] coordinates {(v4) (v5)};
		\draw [firstedge] plot[smooth, tension=.7] coordinates {(v4) (v6)};
		\draw [firstedge] plot[smooth, tension=.7] coordinates {(v4) (v8)};
		\draw [firstedge] plot[smooth, tension=.7] coordinates {(v5) (v6)};
		\draw [firstedge] plot[smooth, tension=.7] coordinates {(v5) (v8)};
		\draw [firstedge] plot[smooth, tension=.7] coordinates {(v6) (v7)};
		\draw [firstedge] plot[smooth, tension=.7] coordinates {(v7) (v8)};

        \draw [secondedge]  [out=-100,in=10,looseness = 1]  (v1)to (v5);
        \draw [secondedge][out=135,in=-30,looseness = 1]  (v5)to (v7);
        \draw [secondedge] [out=-70,in=130,looseness = 1]  (v7)to (v4);
        \draw [secondedge] [out=100,in=-10,looseness = 1.2]  (v4)to (v2);
        \draw [secondedge][out=-70,in=150,looseness = 1]  (v2)to (v6);
        \draw [secondedge] [out=-145,in=35,looseness = 1]  (v6)to (v3);
        \draw [secondedge] [out=25,in=-135,looseness = 1]  (v3)to (v8);
        \draw [secondedge] [out=70,in=-130,looseness = 1]  (v8)to (v1);
  \put(-1,-3.9){ $(7^2,\, 5^6)$ }
	\end{tikzpicture}
 \\ \quad
    \caption{The \ST-connected realizations of $(6^3,\, 5^4)$, $(6^4,\,  5^4)$, $(7,\,  6^2,\,  5^5)$  and $(7^2,\, 5^6)$.}
    \label{FIG: Z3+HC}
\end{figure}
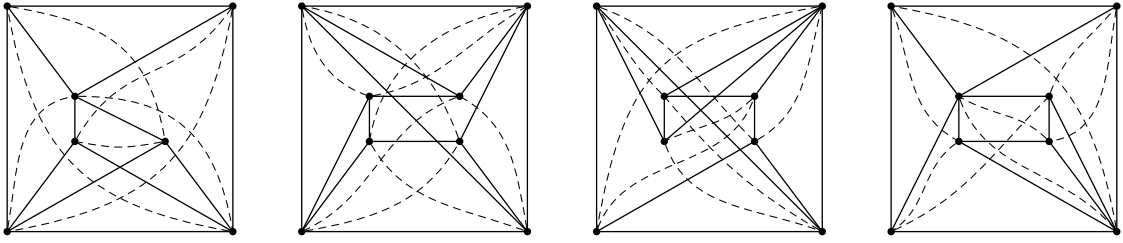

The above special sequences would serve as inductive bases for the proof of Theorem \ref{THM: S3connectedsequence}. In the following, we outline  the main ideas of our proof. We prove Theorem 1.4 by induction on $n$. The general strategy is to apply several operations that efficiently reduce the order of graphic sequences. Then, based on the induction hypothesis, we can obtain a simple \ST-connected graph to construct the desired realization.
These operations include laying off a vertex of degree $4$ or higher, lifting a vertex of degree $4$ or $5$, and contracting a special $K_6$. 
If the sum of degrees is much larger than $6n - 4$, then we lay off a minimum degree vertex directly. 
On the other hand, if the sum of all degrees is close to $6n - 4$, then we perform the lifting operation on a minimum degree vertex. After performing this operation, we need to carefully verify whether the sum of degrees is at least $6(n - 1) - 4$.
And in cases where the sum of degrees equals $6n - 4$ and the minimum degree is 5, laying and lifting operations are ineffective.  However, we can reduce the order of sequences by contracting a special $K_6$, after which the sum of degrees in the resulting reduced sequence is large enough to satisfy the induction hypothesis.
Additionally, for the other degree sequences in which the above operations are not applicable, their realizations are obtained by adding Hamiltonian cycles to certain $\ZT$-connected graphs.

\section{Graphic sequences with at least two large degree vertices} \label{SEC: 2large}
In this section, we focus on graphic sequences that contain at least two vertices of high degree. Specifically, we present $\ST$-connected realizations for these sequences.
\begin{theorem}\label{THM: 2large}
    Let $n\geq7$, and let $\pi=(d_1,\,  \ldots,\, d_{n})\in \GS$  with  $d_1 = d_2 = n-1$ and $ d_n \ge 4$. If $\sum_{i =1}^{n} d_i\geq 6n-4$, then $\pi\in \GS(\ST)$.
\end{theorem}

Before providing the complete proof of Theorem \ref{THM: 2large}, we first need to establish the cases where $n$ is small, as these will form the foundation for our induction.

\begin{lemma}\label{LEM: n<=8_2large}
    Let $n\in\{7,\, 8\}$, and let $\pi=(d_1,\,  \ldots,\, d_{n})\in \GS$  with  $d_1 = d_2 = n-1$ and $ d_n \ge 4$. If $\sum_{i =1}^nd_i \geq 6n-4$, then $\pi\in \GS(\ST)$.
\end{lemma}
\begin{proof}
    If $n = 7$, then, by the condition above, we have $38 \le \sum_{i =1}^7d_i \le 42$, and hence
    
    $$\pi \in \{ (6^3,\, 5^4),\, (6^4,\, 5^2,\, 4),\,    (6^5,\, 5^2),\,    (6^7) \}.$$
    The sequence $\pi=(6^3,5^4)$ belongs to $\GS(\ST)$ by Lemma \ref{LEM: S3realizationbyaddingHC}. Now we give  proofs for the other sequences.
    For $\pi = (6^4,\, 5^2,\, 4)$, let $G$ be the  graph obtained by deleting two adjacent edges from a $K_7$. Clearly, $G$ is a realization of $\pi$ and contains  $K_{6}$ as a proper subgraph. Additionally, $G/K_6 = 4K_2$.
    According to Lemmas \ref{LEM: S3graph} (ii) and \ref{LEM: contractK6 }, $G$ is an \ST-connected realization of $\pi$, i.e., $\pi  \in \GS(\ST)$. 
    For $\pi = (6^5,\, 5^2)$, the sequence has a realization that is isomorphic to $K_7$ with one edge deleted, denoted by $G$. This graph $G$ contains $K_6$ as a proper subgraph and $G/K_6 = 5K_2$. Consequently, $G$ is \ST-connected by Lemmas \ref{LEM: S3graph} (ii) and \ref{LEM: contractK6 }. Therefore, $G$ is an $\ST$-connected realization of $\pi$, indicating that $\pi \in \GS(\ST)$.
     For $\pi = (6^7)$, $K_7$ is a realization of $\pi$. Therefore, by Lemma \ref{LEM: S3graph} (i), we have 
     $\pi  \in \GS(\ST)$.
    
    If $n = 8$, then $44\le \sum_{i =1}^8d_i \le 56$. Now we divide our discussion according to the value of $\sum_{i =1}^8d_i$.
    
    We may suppose $\sum_{i =1}^8d_i = 44$ first.  
    Since $\pi \in \GS$, we have
    $$\pi \in \{(7^4,\,  4^4),\, (7^3,\,  6,\,  5,\,  4^3),\,  (7^3,\,  5^3,\,  4^2),\,  (7^2,\,  6^3,\,  4^3),\, (7^2,\,  5^6),\,  (7^2,\,  6^2,\,  5^2,\,  4^2),\,  (7^2,\,  6,\,  5^4,\,  4)\}.$$
    By Lemmas \ref{LEM: S3realizationbylifting} and   \ref{LEM: S3realizationbyaddingHC}, it only remains to consider the last two cases that $$\pi\in \{(7^2,\, 6^2,\, 5^2,\, 4^2),\, (7^2,\, 6,\, 5^4,\, 4)\}.$$  For these two cases, we have $\pi_l \in \{ (6^4,\, 5^2,\, 4),\, (6^3,\, 5^4)\}$. According to the proof of case $n=7$ above, we find that $\pi_l\in \GS(\ST)$. Therefore, by Lemma \ref{LEM: sequenceinverselift},  $\pi\in \GS(\ST)$. 

    Assume instead that $46\le \sum_{i =1}^8d_i \le 56$. 
   We consider the laying sequence $\pi'$. If $\pi' \in \GS$ satisfies $\sum_{i =1}^7d'_i \ge 38$ and $d'_7\ge 4$, then, by the discussion of $n=7$ above, $\pi'$ has an \ST-connected realization, and so does $\pi$ by Lemma \ref{LEM: sequence_addvertex}. Note that $$ \sum_{i =1}^7d'_i = \sum_{i =1}^7d_i - 2d_8.$$
   For $d_8 \in \{4,\, 6,\,7\}$, we have $\sum_{i =1}^7d'_i\ge \min\{46 - 2\times 4 ,\, 50-2\times6,\, 56 -2\times7 \}=38$. 
   Thus, we are done when $d_8 \in \{4,\, 6,\,7\}$.
   For $d_8 = 5$, we can also get that $\sum_{i =1}^7d'_i\ge 38$ and obtain the \ST-connected realization of $\pi$ expect the situation that  $d_8 = 5$ and $\sum_{i =1}^nd_i = 46$. There are only two remained cases: $\pi \in \{ (7^3,\, 5^5),\,  (7^2,\, 6^2,\, 5^4)\}$, which have the same corresponding lifting sequence $\pi_l=(6^3,\, 5^4)$. According to Lemma \ref{LEM: sequenceinverselift} and the discussion of $n=7$ above, we get $\pi\in \GS(\ST)$. This complete the proof of Lemma \ref{LEM: n<=8_2large}.
\end{proof}

In the cases when there are many low degrees, the induction may not be an effective method for our purpose. Instead, a construction method is employed to prove those cases.
Given any two graphs $G$ and $H$, the graph $G \vee H$ is defined as $V(G \vee H) = V(G) \cup V(H)$  and $$E(G \vee H) = E(G) \cup E(H) \cup \{uv~|~u \in V(G),\,  v \in V(H)\}.$$

\begin{figure}[ht]
\centering
 \begin{tikzpicture}[scale =0.3]
	\tikzstyle{transition}=[circle,draw=black,fill=black,thick,
	inner sep=0pt, minimum  size=0.7mm]
	\tikzstyle{firstedge}=[line width = 0.5pt,color=black]
	\tikzstyle{secondedge}= [densely dashed]
 %first graph
	\node [transition] (v0) at (-1,0) {};
	\node [transition] (v1) at (-1,5) {};
    \node [transition] (v2) at (2.54,3.54) {};
    \node [transition] (v3) at (4,0) {};
    \node [transition] (v4) at (2.54,-3.54) {};
    \node [transition] (v5) at (-1,-5) {};
    \node [transition] (v6) at (-4.54,-3.54) {};
    \node [transition] (v7) at (-6,0) {};
    \node [transition] (v8) at (-4.54,3.54) {};
    \node [transition] (v9) at (-12,2) {};
    \node [transition] (v10) at (-12,-2) {};
    
    \draw [firstedge] plot[smooth, tension=.7] coordinates {(v0) (v1)};    \draw [firstedge] plot[smooth, tension=.7] coordinates {(v0) (v2)};    \draw [firstedge] plot[smooth, tension=.7] coordinates {(v0) (v3)};    \draw [firstedge] plot[smooth, tension=.7] coordinates {(v0) (v4)};    \draw [firstedge] plot[smooth, tension=.7] coordinates {(v0) (v5)};    \draw [firstedge] plot[smooth, tension=.7] coordinates {(v0) (v6)};    \draw [firstedge] plot[smooth, tension=.7] coordinates {(v0) (v7)};    \draw [firstedge] plot[smooth, tension=.7] coordinates {(v0) (v8)};
     \draw [firstedge] plot[smooth, tension=.7] coordinates {(v9) (v10)}; 
     \draw [secondedge] [out=90,in=45,looseness = 1]  (v1)to (v2);
    \draw [secondedge] [out=0,in=-45,looseness = 1]  (v3)to (v4);
     \draw [secondedge] [out=-90,in=-135,looseness = 1]  (v5)to (v6);
    \draw [secondedge] [out=-180,in=135,looseness = 1]  (v7)to (v8);
    \put(0,-2){$...$}
    \put(0,2){$C_1$}
    \put(-3.2,0.7){$C_2$}
    \put(-1.8,-2){$C_3$}
    \put(0.8,-1){$C_{d}$}
    \put(0,0.2){$u$}
    \put(-8,1){$u_1$}
    \put(-8,-1.3){$u_2$}
    \put(1.8,2){$v_{1,\,1}$}
    \put(-1.4,3.8){$v_{1,\,n_1}$}
    \put(-3.3,2.6){$v_{2,\,1}$}
    \put(-4.2,-0.5){$v_{2,\,n_2}$}
    \put(-4,-2){$v_{3,\,1}$}
    \put(-1.3,-4){$v_{3,\,n_3}$}
    \put(1,-2.7){$v_{d,\,1}$}
    \put(2,0.4){$v_{d,\,n_d}$}
    \put(-6,-0.2){$ \vee $}
    \end{tikzpicture}
    \caption{The \ST-connected realization of $\pi=((n-1)^{2},\, d_{3},\, 4^{n-3})$ with $n-1\geq d_{3}\geq10$.}
    \label{FIG: 10^3_4^8}
    \end{figure}
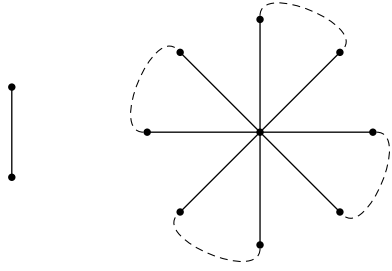

\begin{lemma}\label{LEM: 2large_d4=4}
      If $\pi=((n-1)^{2},\, d_{3},\, 4^{n-3})\in \GS$ with $d_{3}\geq10$,
    then $\pi \in \GS(\ST)$.

\end{lemma}
  
\begin{proof}
    Let $d=\frac{d_{3}-2}{2}$. For $1 \leq i \leq d$, let $C_i$ be a cycle of length $n_i+1$ such that  $V(C_i) = \{u,\, v_{(i,\,1)},\, \ldots, v_{(i,\,n_i)}\}$ and $E(C_i) = \{uv_{(i,\,1)},\,v_{(i,\,1)}v_{(i,\,2)},\, \ldots, v_{(i,\,n_i-1)}v_{(i,\,n_i)},\,v_{(i,\,n_i)}u\}$, where the integer $n_i\ge 2$. Now we suppose $\sum_{i=1}^{d} n_i =n-3$ and $V(C_j)\cap V(C_k)= \{u\}$ for any $j\neq k$.  Then  we construct a new
     graph $C^*$ such that $V(C^*) = \cup_{i=1}^{d}V(C_i)$ and
    $E(C^*) = \cup_{i=1}^{d}E(C_i)$.
    Clearly, $C^*$ is a simple graph, and its degree sequence is $(d_3-2,\, 2^{n-3})$. 
    Next we construct a new graph $G$ such that $G = K_2 \vee C^*$ as depicted in Figure \ref{FIG: 10^3_4^8}, where $V(K_2) = \{u_1,\, u_2\}$. Note that it is a realization of $\pi$. Furthermore, by a series of lifting operations, the graph containing $K_{(1,3,3)}$ as a spanning subgraph with the vertex set $\{u,\,u_1,\,u_2\}$ can be obtained from $G$.  
    The lifting process is  as follows: 
    \begin{align*}
     G  
     &\rightarrow G_{[v_{(1,\, 1)},\, uv_{(1,\, 2)}]} \rightarrow G_{[v_{(1,\, 2)},\, uv_{(1,\, 3)}]} \rightarrow \cdots \rightarrow G_{[v_{(1,\, n_1-1)},\, uv_{(1,\, n_1)}]} \rightarrow G_{[v_{(1,\, n_1)},\, uu_1]}  \\
     &\rightarrow G_{[v_{(2,\, 1)},\, uv_{(2,\, 2)}]}  \rightarrow \cdots \rightarrow G_{[v_{(2,\, n_2-1)},\, uv_{(2,\, n_2)}]} \rightarrow G_{[v_{(2,\, n_2)},\, uu_1]}  \\
     &\rightarrow \cdots \\
     &\rightarrow G_{[v_{(p,\, 1)},\, uv_{(p,\, 2)}]}  \rightarrow \cdots \rightarrow G_{[v_{(p,\, n_p-1)},\, uv_{(p,\, n_p)}]} \rightarrow G_{[v_{(p,\, n_p)},\, uu_1]} \\
     &  \rightarrow G_{[v_{(p+1,\, 1)},\, uv_{(p+1,\, 2)}]} \rightarrow \cdots \rightarrow G_{[v_{(p+1,\, n_{p+1}-1)},\, uv_{(p+1,\, n_{p+1})}]} \rightarrow G_{[v_{(p+1,\, n_{p+1})},\, uu_2]} \\ 
     &  \rightarrow \cdots\\
     &  \rightarrow G_{[v_{(d,\, 1)},\, uv_{(d,\, 2)}]} \rightarrow \cdots \rightarrow G_{[v_{(d,\, n_d-1)},\, uv_{(d,\, n_d)}]} \rightarrow G_{[v_{(d,\, n_d)},\, uu_2]},\end{align*}
     where $p= \lfloor\frac{d_3-2}{4}\rfloor$ and $d= \frac{d_3-2}{2}$. The resulting graph, denoted by $G'$, satisfies $V(G') = \{u,\, u_1,\, u_2\}$, $|E_{G'}(u_1,\,u_2)| =1$, $$|E_{G'}(u,\,u_1)| =\lfloor\frac{d_3-2}{4}\rfloor+1,\, \text{and} ~ |E_{G'}(u,\,u_2)| = \frac{d_3-2}{2}-\lfloor\frac{d_3-2}{4}\rfloor+1.$$
    Since $d_3\ge 10$, 
    we get $$|E_{G'}(u,\,u_1)| \geq 3~ \text{and}~|E_{G'}(u,\,u_2)| \geq 3,$$ and so $K_{(1,3,3)}$ is a spanning subgraph of $G'$. Hence, by Lemma \ref{LEM: S3graph} (iii) and the definition of $\ST$-connectivity, $G'$ is $\ST$-connected. Therefore, by Lemma \ref{LEM: liftinggraph} recursively, $G$ is an $\ST$-connected realization of $\pi$, which means that $\pi  \in \GS(\ST)$.
     %For instance, if $n=11$, then $\pi = (10^3,\,  4^8)  \in \GS(\ST)$ (For more details, refer to Lemma \ref{LEM: S3realizationbylifting}). Therefore, by Lemmas \ref{LEM: S3graph} (iii) and  \ref{LEM: liftinggraph} , $K_2 \vee C^*_{\frac{d_3-2}{2}} \in \SST$, which means that $\pi  \in \GS(\ST)$.
\end{proof}

Now we prove Theorem \ref{THM: 2large}.

\noindent\textbf{Proof of Theorem \ref{THM: 2large}.}
    We argue by induction on $n$. The cases of $7\le n \le 8$ are proved by Lemma \ref{LEM: n<=8_2large}. Now we focus on $n \ge 9$ and divide our discussion according to the value of $d_n$.
    
    We may suppose $d_n =4$ first.   
    If $\sum_{i =1}^nd_i = 6n-4$ and $d_3 \le n-2$, then we consider the lifting sequence $\pi_l$. Note that, we have $\sum_{i =1}^{n-1}d^l_i = 6n-10$ and $n-2 = d^l_{1} = d^l_{2} \ge d^l_{3} \ge \ldots \ge d^l_{n-1}\ge 4$. We also have $f(\pi)=\max \{i~|~d_{i}\geq i,\,  1\le i\le n\}\le 5$, as otherwise $\sum_{i =1}^{n-1}d^l_i > 6n-10$, which leads to a contradiction. 
     Since $d^l_{n-1}\ge 4$, we have $\sum_{i=1}^{k}d^l_i\leq k(n-2) \le k(k-1)+\sum_{i=k+1}^{n-1} \min\{k,\,  d^l_i \}$ for each  integer $ k$ with $1 \le k \le 4$, and so we only need to consider the case of $k = 5 $.
     By $\sum_{i =1}^{n-1}d^l_i -(d^l_1+d^l_2+4(n-3)) = 6$, we know that after assigning $4$ to each of $d_{3}^{l},\dots,d_{n-1}^{l}$, there are a total of $6$ remaining. Consequently, we have 
     $$\sum_{i=1}^{5}d^l_{i}\leq  d^l_1 +d^l_2+  4\times 3 +6  =   2n+14.$$
     Since $d^l_{n-1}\ge 4$, we get 
     $$ 5\times 4 + \sum_{i=6}^{n-1}\min \{5,\, d^l_{i}\} \geq20 +4(n-1-5) = 4n -4.$$
     According to $n\geq 9$, we can verify that 
    $$\sum_{i=1}^{5}d^l_{i}\leq 5\times 4+\sum_{i=6}^{n-1}\min \{5,\, d^l_{i}\}.$$
    Hence Theorem \ref{THM: verify GS-iff condition}   says that $\pi_{l} \in \GS$. Then, by the induction hypothesis, $\pi_{l}$ has an $\ST$-connected realization. Since $2(n-1) - 2 < 6(n-1)-4-2(n-2)$ and by Lemma \ref{LEM: sequenceinverselift}, we conclude that $\pi \in \GS(\ST)$.

    If $\sum_{i =1}^nd_i = 6n-4$ and $d_3 = n-1$, then we have $6n-4 = \sum_{i =1}^nd_i \ge 3(n-1)+4(n-3)$, i.e. $n \le 11$. According to the conditions above, a simple calculation shows that $$\pi \in \{ (8^{3},\, 5^2,\, 4^4),\,  (8^3,\,  6,\,  4^5),\,  (9^3,\,  5,\,  4^6),\,  (10^3,\,  4^8)\}.$$
    It follows from Lemma \ref{LEM: S3realizationbylifting} that $\pi$ has an \ST-connected realization.

    If $\sum_{i =1}^nd_i \ge 6n-2$, then we consider the laying sequence $\pi'$. Note that $\sum_{i =1}^{n-1}d'_i \geq 6(n-1)-4$.  
    The case of $d_4=4$ follows by Lemma \ref{LEM: 2large_d4=4}, and we only need to consider $d_4 \ge 5$. Now $\pi' \in \GS$ by Theorem \ref{THM: vetifyGSdelete}. Since $\sum_{i =1}^{n-1}d'_i \ge 6(n-1)-4$ and $d'_{n-1} \ge 4$, by the induction hypothesis $\pi'$ has an \ST-connected realization. Therefore, by Lemma \ref{LEM: sequence_addvertex}, $\pi$ has an \ST-connected realization as well.

    Assume instead that $d_n \ge 5$. If $\sum_{i =1}^{n}d_i \ge 6n$, then the laying sequence $\pi'$ satisfies $$\sum_{i =1}^{n-1}d'_i \ge \max\{6(n-1)-4+(10-2d_n),\,  2(n-1)+(n-4)d_n\}\ge 6(n-1)-4.$$ Thus $\pi'$ has an \ST-connected realization by the induction hypothesis, and so does $\pi$ by Lemma \ref{LEM: sequence_addvertex}. When $n \geq 12$, we get $\sum_{i =1}^{n}d_i \ge 2(n-1)+ 5(n-2) \ge 6n$, which completes the proof of this case. Hence we only need to consider the situation where $$9\le n\le 11,\,  d_n = 5,\, d_1=d_2=n-1,\,  \text{and}~\sum_{i =1}^{n}d_i \le 6n-2.$$
    Then a simple calculation shows that $\pi \in \{(8^2,\, 6,\, 5^6),\, (9^2,\, 5^8)\}$. Thus the lifting sequence $\pi_l$ satisfies $\pi_l \in \{(7^2,\, 5^6),\,  (8^2,\, 5^6,\,  4)\}$, which is graphic. Notice that $(7^2,\, 5^6) \in \GS(\ST)$ by Lemma \ref{LEM: n<=8_2large} and $(8^2,\, 5^6,\,  4) \in \GS(\ST)$ by the induction hypothesis. Therefore, $\pi \in \GS(\ST)$ by Lemma \ref{LEM: sequenceinverselift}. 
    This completes the proof of Theorem \ref{THM: 2large}.$\blacksquare$

\section{Some critical cases}\label{SEC: boundary}
In this section, we address critical cases where laying and lifting operations are ineffective. To do this, we first need to verify whether the given degree sequence is a graphic sequence. The following lemma is useful for this purpose.

\begin{lemma}\label{LEM: GS_dn>=5_sum=6n-4}
    Let $n\ge7$, and let $\pi=(d_1,\,  \ldots,\, d_n)$ be an integer-valued sequence with $n-1 \ge d_1 \ge \cdots \ge d_n \ge 5$. If $\sum_{i =1}^n d_i= 6n-4$, then $\pi\in \GS$.
\end{lemma}
\begin{proof}
    Since $d_n\ge5$, we have $\sum_{i=1}^{k}d_i\leq k(n-1) \le k(k-1)+\sum_{i=k+1}^{n} \min\{k,\,  d_i \}$ for each  integer $ k$ with $1 \le k \le 5$. Thus when $f(\pi) \le 5$,  we have $\pi \in \GS$ by Theorem \ref{THM: verify GS-iff condition}. Now we consider the case of $f(\pi)\ge 6$.
    
    Suppose by contradiction that $\pi \notin \GS$. Then, according to Theorem \ref{THM: verify GS-iff condition}, there exists an integer $k_0$ with $6\leq k_0\leq f(\pi)$ such that $\sum_{i=1}^{k_0}d_i>k_0(k_0-1)+\sum_{i=k_0+1}^{n}\min\{k_0,\, d_i\} $.  
    Then it follows that
\begin{align}\nonumber
6n -4  = \sum_{i=1}^{k_0}d_i + \sum_{i=k_0+1}^{n}d_i  & > k_0(k_0-1) +\sum_{i=k_0+1}^{n}\min\{k_0,\, d_i\} + \sum_{i=k_0+1}^{n}d_i\\ \nonumber & \ge k_0(k_0-1)+5(n-k_0) + 5(n-k_0)\\\nonumber
& =10n+ (k_0-\frac{11}{2})^2-\frac{121}{4}\\\nonumber
& > 6n -4,
\end{align}
where the last inequality is obtained from $n\ge7$ and $k_0 \ge 6$.
This contradiction implies the truth of Lemma \ref{LEM: GS_dn>=5_sum=6n-4}. 
\end{proof}

In some cases where degree sequences contain many low degrees, we use Lemma \ref{LEM: Z_3+H=S3} and some Hamiltonian properties to obtain $\ST$-connected realizations. We need the following closure concept of Hamiltonian graphs due to Bondy and Chv\'{a}tal \cite{BC1976}.  
To obtain the \textit{closure} of a graph $G$, we repeatedly connect pairs of nonadjacent vertices whose degree sum is at least $|V(G)|$ until no such pair remains in the graph.
Additionally, the following sufficient condition for the existence of Hamiltonian cycle by Bondy and Chv\'{a}tal \cite{BC1976} is essential for the concept of closure.

\begin{lemma}[\cite{BC1976}]\label{LEM: hamiltonian sufficient}
    A simple graph is Hamiltonian if and only if its closure is Hamiltonian.
\end{lemma}

For an integer-valued sequence $\pi=(d_1,\,  \ldots,\, d_n)$ with  $n-1 \ge d_1  \ge \cdots \ge d_n \ge 2$, we define $\pi_{-2} = (d''_1,\, \ldots,\, d''_{n}) = (d_1-2,\,  \ldots,\, d_{n}-2)$. The following observation follows from Lemma \ref{LEM: Z_3+H=S3} immediately.

\begin{observation}\label{OB: Z3c}
    Let $\pi\in \GS$. If $\pi_{-2}\in \GS(\ZT)$ has a $\ZT$-connected realization $G$ such that the complement of $G$ is Hamiltonian, then $\pi \in \GS(\ST)$.
\end{observation}
\begin{proof}
    Let $C$ be a Hamiltonian cycle in the complement of $G$. Then we construct a new graph $G_1$ from $G$ and $C$ by setting $V(G_1)=V(G)$ and $E(G_1) = E(G) \cup E(C)$. Then, by Lemma \ref{LEM: Z_3+H=S3}, $G_1$ is an \ST-connected realization of $\pi$, i.e., $\pi \in \GS(\ST)$.
\end{proof}

\begin{lemma}\label{LEM: sequenceZ3+H=S3}
    Let $n\geq8$, and let $\pi=(d_{1},\, d_{2},\, d_{3},\,  5^{n-3})\in \GS$ with $d_{1}\leq n-1$,\,  $d_{2}\leq n-2$, and $d_{3}\leq n-3$. If $\pi_{-2} \in \GS(\ZT)$, then $\pi\in \GS(\ST)$.
\end{lemma}
\begin{proof}
Let $G'$ be a $\ZT$-connected realization of $\pi_{-2}$. 
The complement of graph $G'$, denoted as $\overline{G'}$, has degree sequence $\overline{{\pi}_{-2}} = ((n-4)^{n-3},\,  n+1-d_{3},\,  n+1-d_{2},\,  n+1 -d_{1})$. Since $n+1-d_{3} + (n-4)\geq n$, $n+1-d_{2} + (n-3)\geq n$, and $n+1-d_{1} + (n - 2)\geq n$, we can form the closure of $\overline{G'}$ by connecting all vertices of degree $(n-4)$ in $\overline{G'}$, connecting the vertex of degree $(n+1-d_{3})$ to the vertices of degree $(n-4)$ in $\overline{G'}$, and connecting the remaining vertices to the large ones. Then the closure of $\overline{G'}$ is $K_n$. 
Hence, by Lemma \ref{LEM: hamiltonian sufficient}, $\overline{G'}$ is Hamiltonian. By Observation \ref{OB: Z3c}, we have $\pi \in \GS(\ST)$.
\end{proof}

Now, we are ready to establish the main result of this section.

\begin{theorem}\label{THM: dn=5_sum=6n-4}
    Let $n\geq7$, and let $\pi=(d_{1},\, \ldots,\, d_{n})\in \GS$ with $  d_n \ge 5$.
    If $\sum^n_{i=1} d_i = 6n-4$, then $\pi\in \GS(\ST)$.
\end{theorem}

\begin{proof}
 We argue by induction on $n$. 
If $n =7$, then $d_1 = d_2 =6$, and thus $\pi \in \GS(\ST)$ by Lemma \ref{LEM: n<=8_2large}. According to Theorem \ref{THM: 2large}, we only need to consider the case where $n \ge 8$ and $d_2 \le n-2$.

We may suppose $d_1=6$ first. This means that $\pi=(6^{n-4},\, 5^4)$.  If $n \in \{ 8,\,  9\}$, then $\pi \in \{ (6^{4},\, 5^4),\,  (6^{5},\, 5^4)\}$. Hence $ \pi \in \GS(\ST)$ by Lemma \ref{LEM: S3realizationbyaddingHC}.  If $n \ge 10$, then we consider $\pi_{-2} = (4^{n-4},\, 3^4)$ and the sum of its degrees  is  $4n -4$.  Since $10\ge \frac{1}{3}\lfloor \frac{(4+3+1)^2}{4}\rfloor$ and by Lemma \ref{LEM: verifyGSinequality}, $\pi_{-2}$ is graphic. 
Furthermore, according to Theorem \ref{THM: Z3connectedsequence}, we have  $\pi_{-2} \in \GS(\ZT)$. Let $G'$ be a  $\ZT$-connected realization of $\pi_{-2}$, and let $\overline{G'}$ be the complement of $G'$. Then the degree sequence of $\overline{G'}$ is  $( (n-4)^4,\,(n-5)^{n-4})$. Since $n \ge 10$, the closure of $\overline{G'}$ is $K_n$, and so $\overline{G'}$ is Hamiltonian by Lemma \ref{LEM: hamiltonian sufficient}. Therefore, by Observation \ref{OB: Z3c}, we have $\pi \in \GS(\ST)$.

Assume instead that $d_1\ge7$. If $n = 8$, then we have $\pi = (7,\,  6^{2},\, 5^5) \in \GS(\ST)$ by Lemma \ref{LEM: S3realizationbyaddingHC}.  If $n \in \{9,\, 10,\, 11\}$, then $d_3\leq n-3$ by calculating the sum of degrees. Then we consider $\pi_{-2} = (d''_1,\, \ldots,\, d''_{n})$. Note that the sum of its degrees is $4n-4$. Since $\sum_{i =1}^{n} d_i = 6n -4$ and $d_1 \ge 7$, we have $d_4 \le 6$. Thus $f(\pi_{-2}) \le 4$.
     Since  $\sum_{i=1}^{k}d''_i  \le k(k-1)+\sum_{i=k+1}^{n} \min\{k,\,  d''_i \}$ for  $1\leq k\leq f(\pi_{-2})$, by Theorem  \ref{THM: verify GS-iff condition} we conclude that $\pi_{-2} \in \GS$. Furthermore, according to Theorem \ref{THM: Z3connectedsequence}, $\pi_{-2} \in \GS(\ZT)$.  
     For $d_4 =5$, we have $\pi \in \GS(\ST)$ by Lemma \ref{LEM: sequenceZ3+H=S3}.
     For $d_4 = 6$,  we consider $\pi_{-2} = (d_1-2,\,  d_2-2,\,  d_3-2,\,  4,\,  \ldots,\,  3)$, where the dots imply that the omitted degrees take values between $3$ and $4$.
     Since $d''_{4}$ is small, it is easy to check that $\pi_{-2} \in \GS(\ZT)$ by Theorems \ref{THM: verify GS-iff condition} and \ref{THM: Z3connectedsequence}. Let $G'$ be a $\ZT$-realization of $\pi_{-2}$. Then the degree sequence of $\overline{G'}$ is $ (n+1-d_1,\,  n+1-d_2,\,  n+1-d_3,\,  n-5,\, \ldots,\, n-4)$. Given that $n \ge 9$, the closure of $\overline{G'}$ is $K_n$, and so $\overline{G'}$ is Hamiltonian by Lemma \ref{LEM: hamiltonian sufficient}. Then it follows from Observation \ref{OB: Z3c} that $\pi \in \GS(\ST)$.

     If $n \ge 12$ and $d_3 \ge n-5$, then $n \le 13$ as $3(n-5) -15 \le d_1+d_2+d_3-15 \le 6n-4-5n$.
     Hence, by $\pi \in \GS$ and $\sum^n_{i=1} d_i = 6n-4$, we have
     $$\pi \in \{(9,\, 7^2,\, 5^9),\,  (8^2,\,  7,\,  5^9),\, (8,\, 7^2,\, 6,\, 5^8),\, (7^4,\, 5^8),\, (7^3,\, 6^2,\, 5^7),\,  (8^3,\, 5^{10})\}.$$
     Now we consider $\pi_{-2}$.
     Since $12 \ge \frac{1}{3}\lfloor \frac{(7 +3+1)^2}{4}\rfloor$ and by Lemma \ref{LEM: verifyGSinequality}, we have $\pi_{-2} \in \GS$. Hence, by Theorem \ref{THM: Z3connectedsequence},  $\pi_{-2} \in \GS(\ZT)$. Let $G'$ be a $\ZT$-realization of $\pi_{-2}$. From its degree sequence, we can determine that the closure of $\overline{G'}$ is  $K_n$, which means that $\overline{G'}$ is Hamiltonian by Lemma \ref{LEM: hamiltonian sufficient}. Hence $\pi \in \GS(\ST)$ by Observation \ref{OB: Z3c}.

If $n \ge 12$ and $d_3 \le n-6$, then we consider three cases according to the value of $d_1+ d_2 -10$. 
\begin{itemize}
    \item For $d_1+ d_2 -10 \ge n-5$, we have $\pi \in \{(d_1,\, d_2,\, 6,\, 5^{n-3}),\,  (d_1,\, d_2,\, 5^{n-2})\}$, since $d_i \ge 5$, $\sum_{i=1}^nd_i = 6n-4$ and $d_1+ d_2 -10 \ge n-5$. Now we consider $\pi_{-2}$.
    Since $f(\pi_{-2}) \le 3$ and by Theorem  \ref{THM: verify GS-iff condition}, a simple calculation shows $\pi_{-2} \in \GS$. Hence, by Theorem \ref{THM: Z3connectedsequence}, $\pi_{-2} \in \GS(\ZT)$. Therefore, based on Lemma \ref{LEM: sequenceZ3+H=S3}, $\pi \in \GS(\ST)$.
    \item  For $5 \le d_1+d_2-10 \le n-6$, we have $\pi= (d_1,\,  \ldots,\, d_{n-4},\, 5^4)$ since the sum of the  degrees is $6n-4$. Now let $ \pi^*$ be the non-increasing reordered sequence
    of $(d_1+d_2-10,\,  d_3,\, ...,\, d_{n-4})$. Since  $\sum_{i=1}^{n-4}d_i-10 = 6(n-5)-4$ and $d_3 \le n-6$,  by Lemma \ref{LEM: GS_dn>=5_sum=6n-4} we have  $\pi^* \in \GS$. Then, by the induction hypothesis, $\pi^*$ has an $\ST$-connected realization, denoted by $G'$.  
    Let $u$ be the vertex of $G'$  with  degree $d_1+d_2-10$, and let $ \{u_1,\, \ldots,\,  u_{d_1+d_2-10}\}$ be the set of neighbours of $u$. We also consider the graph $K_6$ with vertex set $V(K_6) = \{v_1,\ldots, v_6\}$. Next
    we construct a new graph $G$ from $G'$ and $K_6$ by setting $V(G) = V(G') \cup V(K_6)-\{u\}$ and  $$E(G) = E(G') \cup E(K_6) \cup \{v_1u_1,\,  \ldots,\,  v_1u_{d_1-5}\} \cup \{v_2u_{d_1-4},\,  \ldots,\,  v_2u_{d_1+d_2-10}\}-E(u).$$ Obviously, $\pi$ is the degree sequence of $G$. Since $\sum_{i =1}^n d_i = 6n-4$ and $d_1\le n-1$, we have $d_2 -5 >0$, and so $K_6$ is a proper subgraph of $G$. 
    Therefore, based on $G/K_6 = G' \in \SST$ and by Lemma \ref{LEM: contractK6 }, $G$ is an \ST-connected realization of $\pi$.
    
    \item For $d_1+ d_2 -10 \le 4$, 
     let $\pi^*$ be the non-increasing reordered sequence of $(d_1+\cdots +d_k-5k,\,  d_{k+1},\, \ldots,\,  d_{n-6+k})$.
    Here, $k$ is the minimum value of $i$ such that $3\le i\le 5$ and $d_1+\cdots +d_i-5i \ge 5$. 
    Since $\sum_{i=1}^{n-6+k}d_i - 5k =6(n-5) -4$ and by Lemma \ref{LEM: GS_dn>=5_sum=6n-4}, we have $\pi^* \in \GS$.
    Hence, by the induction hypothesis, $\pi^*$ has an $\ST$-connected realization, denoted by $G'$.  
    Let $u$ be the vertex of $G'$ with degree $d_1+\cdots +d_k-5k$, and let $ \{u_1,\, \ldots,\,  u_{d_1+\cdots +d_k-5k}\}$ be the set of neighbours of $u$. Furthermore, we consider the graph $K_6$ with vertex set $V(K_6) = \{v_1,\ldots, v_6\}$. Next we construct a new graph $G$ based on $G'$ and $K_6$ by setting  
    $V(G) = V(G') \cup V(K_6)-\{u\}$ 
    and
    \[E(G) = E(G') \bigcup E(K_6)  \bigcup\limits_{i=1}^k \{v_iu_{d_1+\cdots+d_{i-1}-5(i-1)+1},\,  \ldots,\,  v_iu_{d_1+\cdots +d_i-5i}\}-E(u).\]
    Here we agree that when $i=1$, $d_1+\cdots+d_{i-1}-5(i-1)+1=1$.
    Clearly, $G$ is a realization of $\pi$.
    By the definition of $k$, we have $d_{k} -5 >0$, which implies that $K_6$ is a proper subgraph of $G$. 
    Therefore, as $G/K_6 = G' \in \SST$ and by Lemma \ref{LEM: contractK6 }, $G$ is an \ST-connected realization of $\pi$.
    \end{itemize}
    This completes the proof of Theorem \ref{THM: dn=5_sum=6n-4}.
    \end{proof}

\section{Proof of Theorem \ref{THM: S3connectedsequence}}\label{SEC: main theorem}

 The degree sum condition in Theorem \ref{THM: S3connectedsequence} can be illustrated by the following lemma, which is a necessary condition for \ST-connectivity established in \cite{LLW2020}.
\begin{lemma}[\cite{LLW2020}]\label{LEM: S_3property}
    Let $G$ be a nontrivial graph. If $G \in \SST$, then $|E(G)|\geq 3|V(G)|-2$.
\end{lemma}
If $\pi$ has an $\ST$-connected realization $G$, then, according to Lemma \ref{LEM: S_3property}, we have $\sum_{i =1}^nd_i\geq 6n-4$.
It is also straightforward to observe that $\min \{d_1,\,    \ldots,\, d_n\} \ge 4$. For otherwise, if $G$ contains a vertex $v$ with $d(v) \leq 3$, then we assign $\beta(v)=3-d(v)$. Under this boundary condition, all edges incident to $v$ are oriented either all-in or all-out, which means the orientation cannot be strongly connected. This proves the necessity of Theorem \ref{THM: S3connectedsequence}.

Now, we are ready to complete the proof of sufficiency for Theorem \ref{THM: S3connectedsequence}.

We argue by induction on $n$.
Assume that $\sum_{i =1}^nd_i\geq 6n-4$ and $\min \{d_1,\,    \ldots,\, d_n\} \ge 4$. This implies that $n \ge 7$ since $\pi \in \GS$. If $n =7$, then $d_1 = d_2 =6$. Hence,  by Theorem \ref{THM: 2large}, we have $\pi \in \GS(\ST)$ and we only need consider that  $n \ge 8$ with $d_2 \le n-2$. We divide our discussion according to the value of $d_n$.

    If $d_{n}\geq 7$, then we consider the laying sequence $\pi'$. 
    Then, by Theorem \ref{THM: vetifyGSdelete}, $\pi'\in \GS$. Since $\sum_{i=1}^nd_i \geq nd_{n} \geq 7n$,
    We also have  $$\sum_{i=1}^{n-1}d'_i \ge nd_{n}-2d_{n}=(n-2)d_{n}\geq 6(n-1)-4.$$ 
    Hence, by the induction hypothesis, $\pi'$ has an $\ST$-connected realization. Therefore, by Lemma \ref{LEM: sequence_addvertex}, we have $\pi \in \GS(\ST)$.
     
    If $d_{n}=6$ and  $\sum_{i=1}^nd_i\geq 6n+2$, then the case is complete by considering the laying sequence $\pi'$ and using the induction hypothesis. Assume instead that  $d_{n}=6$ and $\sum_{i=1}^nd_i\leq 6n$. Then we have $\pi=(6^n)$. Now we consider the lifting sequence $\pi_{l}=(6^{n-5},\, 5^4)$. Using  Lemma \ref{LEM: GS_dn>=5_sum=6n-4} and by the induction hypothesis, we obtain that $\pi_{l}$ is graphic and belongs to $\GS(\ST)$. 
   For the case of $n=8$, we have $\pi_{l}=(6^3,\, 5^4)$. We denote $G'$ as the $\ST$-connected realization of $(6^3,\, 5^4)$ shown in Figure \ref{FIG: Z3+HC}. Let $u_1$, $u_2$, and $u_3$ be the vertices of $G'$ with degree $6$, and let $v_1,\, \ldots,\, v_4$ be the vertices of $G'$ with degree $5$. We introduce a new vertex $w$ and construct a new graph $G$ by setting  
    $$V(G) = V(G') \cup \{w\}~ \text{and}~E(G) = E(G') \cup \{wu_1,\,wu_2,\,wv_1,\,\dots,\,wv_4\} -\{u_1u_2\}.$$
    Therefore, based on Lemma \ref{LEM: liftinggraph} and the fact that $G_{[w,\,u_1u_2]} = G' \in \SST$, $G$ is an $\ST$-connected realization of $\pi$. This implies that $\pi \in \GS(\ST)$. For the cases of $n\geq 9$, as $4 \times 6 - (6-2) < (n-5) \times 6$ and by Lemma \ref{LEM: sequenceinverselift}, we conclude that $\pi \in \GS(\ST)$.

    If $d_{n}=5$, then we consider the value of $\sum_{i=1}^nd_i$.
    For $\sum_{i=1}^{n} d_i\geq 6n$, we consider the laying sequence $\pi'$. Then, by the induction hypothesis and as Lemma \ref{LEM: sequence_addvertex}, we have $\pi\in \GS(\ST)$. 
    For $\sum_{i=1}^nd_i=6n-2$ and $d_{3}=5$, we determine $\pi=(d_{1},\, d_{2},\, 5^{n-2})$ and $\pi_{-2}=(d_{1}-2,\, d_{2}-2,\, 3^{n-2})$.
    Using Theorems \ref{THM: Z3connectedsequence} and \ref{THM: verify GS-iff condition}, we can obtain that $\pi_{-2}$ is graphic and belongs to $\GS(\ZT)$. 
    Therefore, by Lemma \ref{LEM: sequenceZ3+H=S3}, we obtain that $\pi \in \GS(\ST)$. 
    For $\sum_{i=1}^nd_i= 6n-2$  and $d_{3}\geq 6$, we consider the lifting sequence $\pi_l$, which satisfies $\sum_{i=1}^{n-1}d^l_i=6(n-1)-4$ and $\min \{d^l_1,\, d^l_2,\, \ldots,\,  d^l_{n-1}\} \geq 5$.
    Based on Lemma \ref{LEM: GS_dn>=5_sum=6n-4}, we have $\pi_{l}\in \GS$. 
    Then, by the induction hypothesis, we can conclude that $\pi_{l} \in \GS(\ST)$.
    Since $d_{1}+d_{2}+d_{3}-(5-2)\leq 3n-8\leq 3n+1 \le d_{4}+d_{5}+\cdots +d_{n-1}$,  we have $\pi \in \GS(\ST)$ by Lemma \ref{LEM: sequenceinverselift}. 
    For $\sum_{i=1}^nd_i=6n-4$, we have $\pi \in \GS(\ST)$ based on Theorem \ref{THM: dn=5_sum=6n-4}.
    
    If $d_{n}=4$, then $d_2 \ge 5$, as otherwise $\sum_{i=1}^nd_i\leq (n-1)+4\times (n-1)=5n-5<6n-4$, which leads to a contradiction.
    The following enumerates all possible cases of $\pi$.
    \begin{itemize}
        \item For $\sum_{i=1}^{n}d_i\geq 6n-2$ and $d_{4}\geq5$, we consider the laying sequence $\pi'$. Then, by the induction hypothesis and as Lemma \ref{LEM: sequence_addvertex}, we conclude that $\pi\in \GS(\ST)$.
        
        \item For $\sum_{i=1}^nd_i\geq 6n-2$ and $d_{4}=4$, we consider the lifting sequence $\pi_{l}$. According to  $d_2 \ge 5$, we have $\min\{d_1^l,\,\ldots,\, d_{n-1}^l\} \ge 4$.
        Since $f(\pi_{l})=4$, a simple calculation shows that $\pi_{l}$ is graphic by Theorem \ref{THM: verify GS-iff condition}.
        Furthermore, we obtain $\pi_{l} \in \GS(\ST)$ based on the fact that $\sum_{i=1}^{n-1}d^l_i \geq 6(n-1)-4$ and by the induction hypothesis. 
        Therefore,
        as $d_{1}+d_{2} - (4-2) \leq 2(n-2)<4n-3 \le d_3+d_4+\cdots+d_{n-1}$ and by Lemma \ref{LEM: sequenceinverselift}, we have $\pi \in \GS(\ST)$.
        
        \item  For $\sum_{i=1}^nd_i = 6n-4$, 
        we consider the lifting sequence $\pi_{l}$, and  it satisfies $\sum_{i=1}^{n-1}d^l_i = 6(n-1)-4$.
        Since $d_2-1 \ge 5-1$, we have $\min\{d_1^l,\,\ldots,\, d_{n-1}^l\} \ge 4$. If $\pi_l$ is graphic, then,  by the induction hypothesis  and using Lemma \ref{LEM: sequenceinverselift}, we have $\pi\in \GS(\ST)$. 
        Assume instead that  $\pi_l \not \in \GS$.  
        Then, according to Theorem \ref{THM: vetifyGSdelete} and $d_n =4$, $\pi_l \not \in \GS$ if and only if $(d_{1},\,  \ldots,\, d_{n-1},\, 2)\notin \GS$.
        For convenience, we denote $\pi'' = (d_{1},\,  \ldots,\, d_{n-1},\, 2)$.
        Furthermore, Theorem  \ref{THM: verify GS-iff condition} states that there exists an integer $k$ such that $1 \le k \le f(\pi'')$  and $\sum_{i=1}^{k}d_{i}>k(k-1)+\sum_{i=k+1}^{n-1}\min \{k,\, d_{i}\}+\min \{k,\, 2\}$.
        Clearly, by the fact that $\pi\in \GS$, it is evident that $k \notin \{1,\,2\}$. Additionally, since $d_{2},\, d_{3},\, d_{4}\leq n-2$, it can be concluded that  $k\notin \{3,\, 4\}$. Thus, our attention now turns to the cases where $k \geq 5$.
        Now 
        \begin{align*}
        6n-4 = \sum_{i=1}^{n}d_{i} &= \sum_{i=1}^{k}d_{i} + \sum_{i=k+1}^{n}d_{i}\\
        &>k(k-1)+4(n-1-k) +2 + 4(n-k) \\
        &=(6n-4)+(2n-18)+(k-4)(k-5).    
        \end{align*} Note that the last inequality holds only if $k=5$ and $n=8$. Hence we have  $\sum_{i=1}^{8}d_{i} =44$ and $d_{5}\geq 5$. 
        Furthermore, we get $d_{6}=4$, as otherwise  
        \begin{align*}
        \sum_{i=1}^{8}d_{i} = \sum_{i=1}^{5}d_{i} +\sum_{i=6}^{8}d_{i} &> 5\times 4+\sum_{i=6}^{7}\min\{5,\, d_{i}\}+2 +\sum_{i=6}^{8}d_{i}\\ 
        &\ge (20+5+4+2)+(5+2\times4) \\
        &=44,
        \end{align*}
         which is a contradiction to $\sum_{i=1}^{8}d_{i} =44$. Thus $\pi=(d_{1},\, \ldots,\, d_{5},\, 4^3)$ with $d_{1}+\cdots+d_{5}=32$. This implies that $d_{1}=d_{2}=7$. However, this contradicts  the fact that $d_2 \le n-2 =6$.
    \end{itemize}
  This completes the proof of Theorem \ref{THM: S3connectedsequence}.

\section*{Acknowledgement}
Hong-Jian Lai is supported by National Natural Science Foundation of China (No. 12471333). Jiaao Li is supported by National Key Research and Development Program of China (No. 2022YFA1006400), National Natural Science Foundation of China (Nos. 12222108, 12131013), Natural Science Foundation of Tianjin (No. 22JCYBJC01520), and the Fundamental Research Funds for the Central Universities, Nankai University.

%\bibliographystyle{elsarticle-num}%plain
%\bibliography{ref}
\begin{spacing}{0.7}

\end{spacing}

\end{document}